\documentclass[11pt]{amsart}
\usepackage{mathrsfs,amssymb}
\usepackage{verbatim}
\usepackage{amsfonts}
\usepackage{amsmath}
\usepackage{amsthm}
\usepackage{relsize}
\usepackage{color}
\usepackage[mathscr]{euscript}
\usepackage[rightcaption]{sidecap}
\usepackage{graphicx}
\usepackage{mmacells}
\newtheorem{theorem}{Theorem}[section]
\newtheorem{lemma}[theorem]{Lemma}

\newtheorem{conjecture}[theorem]{Conjecture}

\lstdefinestyle{mystyle}{
    backgroundcolor=\color{backcolour},   
    commentstyle=\color{codered},
    keywordstyle=\color{orange},
    numberstyle=\tiny\color{codegray},
    stringstyle=\color{codegreen},
    basicstyle=\ttfamily\footnotesize,
    breakatwhitespace=false,         
    breaklines=true,                 
    captionpos=b,                    
    keepspaces=true,                 
    numbers=left,                    
    numbersep=5pt,                  
    showspaces=false,                
    showstringspaces=false,
    showtabs=false,                  
    tabsize=2
}

\newcommand{\Q}{\mathbb{Q}}

\newcommand{\R}{\mathbb{R}}

\newcommand{\Z}{\mathbb{Z}}
\newcommand{\A}{\mathfrak{a}}
\newcommand{\B}{\mathfrak{b}}

\newcommand{\C}{\mathbb{C}}
\newcommand{\p}{\mathfrak{p}}

\begin{document}

\title[Gaussian Primes Across Sectors]{A Refined Conjecture for the Variance of Gaussian Primes Across Sectors}
\author[RC, YK,  JL, SM,  AS, SS, EW, EW, JY]{Ryan C. Chen, Yujin H. Kim,  Jared D. Lichtman,\\
Steven J. Miller,  Alina Shubina, Shannon Sweitzer,\\
Ezra Waxman, Eric Winsor, Jianing Yang}
\date{\today}

\address{Department of Mathematics, Princeton University, Princeton, NJ 08544}
\email{rcchen@princeton.edu}

\address{Department of Mathematics, Columbia University, New York, NY 10027}
\email{yujin.kim@columbia.edu}

\address{Department of Mathematics, Dartmouth College, Hanover, NH 03755}
\email{jared.d.lichtman@gmail.com}

\address{Department of Mathematics and Statistics, Williams College, Williamstown, MA 01267}
\email{sjm1@williams.edu}

\address{Department of Mathematics and Statistics, Williams College, Williamstown, MA 01267}
\email{as31@williams.edu}

\address{Department of Mathematics, University California, Riverside, CA 92521}
\email{sswei001@ucr.edu}

\address{Charles University, Faculty of Mathematics and Physics, Department of Algebra, So\-ko\-lov\-sk\' a~83, 18600 Praha~8, Czech Republic}
\email{ezrawaxman@gmail.com}

\address{Department of Mathematics, University of Michigan, Ann Arbor, MI 48109}
\email{rcwnsr@umich.edu}

\address{Department of Mathematics, Colby College, Waterville, ME 04901}
\email{jyang@colby.edu}

\numberwithin{equation}{section}
\maketitle
\begin{abstract}
We derive a refined conjecture for the variance of Gaussian primes across sectors, with a power saving error term, by applying the $L$-functions Ratios Conjecture.  We observe a bifurcation point in the main term, consistent with the Random Matrix Theory (RMT) heuristic previously proposed by Rudnick and Waxman.  Our model also identifies a second bifurcation point, undetected by the RMT model, that emerges upon taking into account lower order terms.  For sufficiently small sectors, we moreover prove an unconditional result that is consistent with our conjecture down to lower order terms.
\end{abstract}
\maketitle

\section{Introduction}
Consider the ring of Gaussian integers $\mathbb{Z}[i]$, which is the ring of integers of the imaginary quadratic field $\Q(i)$.  Let $\A = \langle \alpha \rangle$ be an ideal in $\Z[i]$ generated by the Gaussian integer $\alpha \in \Z[i]$.  The \textit{norm} of the ideal $\A$ is defined as $N(\A) := \alpha \cdot \overline{\alpha}$, where $\alpha \mapsto \overline{\alpha}$ denotes complex conjugation.  Let $\theta_{\alpha}$ denote the argument of $\alpha$.  Since $\Z[i]$ is a principal ideal domain, and the generators of $\A$ differ by multiplication by a unit $\{\pm 1, \pm i\} \in \Z[i]^{\times}$, we find that $\theta_{\A} := \theta_{\alpha}$ is well-defined modulo $\pi/2$.  We may thus fix $\theta_{\A}$ to lie in $[0,\pi/2)$, which corresponds to choosing a generator $\alpha$ that lies within the first quadrant of the complex plane.\\
\\
We are interested in studying the angular distribution of $\{\theta_{\p}\} \in [0,\pi/2)$, where $\p\subsetneq \Z[i]$ are the collection of prime ideals with norm $N(\p)\leq X$.  To optimize the accuracy of our methods, we employ several standard analytic techniques.  In particular, we count the number of angles lying in a short segment of length $1/K$ in $[0,\pi/2]$ using a smooth window function, denoted by $F_{K}(\theta)$, and we count the number of ideals $\A$ with norm $N(\A)\leq X$ using a smooth function, denoted by $\Phi$.  We moreover count prime ideals using the weight provided by the Von Mangoldt function, defined as $\Lambda(\A) = \log N(\p)$ if $\A = \p^{r}$ is a power of a prime ideal $\p$, and $\Lambda(\A) = 0$ otherwise.
\\
\\
Let $f \in C_{c}^{\infty}(\R)$ be an even, real-valued window function.  For $K \gg 1$, define
\begin{equation}\label{Fk}
F_{K}(\theta)\ := \ \sum_{k \in \Z}f\left(\frac{K}{\pi/2}\left(\theta-\frac {\pi}{ 2}\cdot k\right)\right),
\end{equation}
which is a $\pi/2$-periodic function whose support in $[0,\pi/2)$ is on a scale of $1/K$.  The Fourier expansion of $F_{K}$ is given by
\begin{equation}\label{fourapprox}
F_{K}(\theta) = \sum_{k \in \Z}\widehat{F}_{K}(k)e^{i4k\theta},\hspace{10mm} \widehat{F}_{K}(k)=\frac{1}{K}\widehat{f}\left(\frac{k}{K}\right),
\end{equation}
where the normalization is defined to be $\widehat{f}\left(y \right):=\int_{\R}f(x)e^{-2\pi i y x}dx.$\\
\\
Let $\Phi \in C_{c}^{\infty}(0,\infty)$ and denote the \textit{Mellin transform} of $\Phi$ by 

\begin{equation}
\tilde{\Phi}(s)\ := \ \int_{0}^{\infty}\Phi(x)x^{s-1}dx.
\end{equation}
Define
\begin{equation}\label{pisum}
\psi_{K,X}(\theta)
:=\sum_{\A} \Phi \left(\frac{N(\A)}{X}\right) \Lambda(\A) F_K(\theta_{\A}  -\theta),
\end{equation}
where $\A$ runs over all nonzero ideals in $\Z[i]$.  We may then think of $\psi_{K,X}(\theta)$ as a smooth count for the number of prime power ideals less than $X$ lying in a window of scale $1/K$ about $\theta$.  As in Lemma 3.1 of \cite{RudWax}, the mean value of $\psi_{K,X}(\theta)$ is given by

\begin{align} \label{expvalue}
\begin{split}
\langle\psi_{K,X}\rangle &:= \frac {1}{\pi/2}\int_{0}^{\frac \pi 2}\sum_{\A} \Phi \left(\frac{N(\A)}{X}\right) \Lambda(\A) F_K(\theta_{\A}  -\theta)d\theta \sim \frac{X}{K}\cdot C_{\Phi}\cdot c_{f},
\end{split}
\end{align}
where 
\begin{equation}\label{mean value constants}
c_{f} := \frac{1}{4\pi^2}\int_{\R}f(x)dx, \hspace{5mm} \textnormal{ and } \hspace{5mm} C_{\Phi}:= 4\pi^{2}\int_{0}^{\infty}\Phi(u)du.
\end{equation}
For fixed $K>0$, then a smooth version of a result from Hecke \cite{Hecke} states that in the limit as $X \rightarrow \infty$,

\begin{equation}\label{shrink}
\psi_{K,X}(\theta) \sim \frac{X}{K}\cdot c_{f}\cdot C_{\Phi}.
\end{equation}
Alternatively, one may study the behavior of $\psi_{K,X}(\theta)$ for shrinking intervals, i.e. for large $K$.  It follows from the work of Kubilius \cite{Kubilius 1955} that under the assumption of the Grand Riemann Hypothesis (GRH), (\ref{shrink}) continues to hold for $K \ll X^{1/2-o(1)}$.\\
\\
In this paper, we wish to study
\begin{align}
\textnormal{Var}(\psi_{K,X}) &:=  \frac{1}{\pi/2}\int_{0}^{\frac \pi 2}\bigg|\psi_{K,X}(\theta) - \langle\psi_{K,X}\rangle\bigg|^{2}d\theta.
\end{align}
Such a quantity was investigated by Rudnick and Waxman \cite{RudWax}, who, assuming GRH, obtained an upper bound for $\textnormal{Var}(\psi_{K,X})$.\footnote{See also \cite{SarnakUri}.}  They then used this upper bound to prove that almost all arcs of length $1/K$ contain at least one angle $\theta_{\mathfrak p}$ attached to a prime ideal with $N(\mathfrak p)\leq K(\log K)^{2+o(1)}$.\\
\\
Montogomery \cite{Montgomery} showed that the pair correlation of zeros of $\zeta(s)$ behaves similarly to that of an ensemble of random matrices, linking the zero distribution of the zeta function to eigenvalues of random matrices.  The Katz-Sarnak density conjecture \cite{KatzSarn1,KatzSarn2} extended this connection by relating the distribution of zeros across families of $L$-functions to eigenvalues of random matrices.  Random matrix theory (RMT) has since served as an important aid in modeling the statistics of various quantities associated to $L$-functions, such as the spacing of zeros \cite{Hedjhal, Odlyzko, Sound}, and moments of $L$-functions \cite{ConreyIII, ConreyIV}.  Motivated by a suitable RMT model for the zeros of a family of Hecke $L$-functions, as well as a function field analogue, Rudnick and Waxman conjectured that

\begin{equation}\label{RMT Conjecture}
\textnormal{Var}(\psi_{K,X})  \sim \int_{\R}f(y)^{2}dy \int_{0}^{\infty}\Phi(x)^{2}dx\cdot \textnormal{min} (\log X, 2\log K).
\end{equation}

Inspired by calculations for the characteristic polynomials of matrices averaged over the compact classical groups, Conrey, Farmer, and Zirnbauer \cite{Conrey, ConreyII} further exploited the relationship between $L$-functions and random matrices to conjecture a recipe for calculating the ratio of a product of shifted $L$-functions averaged over a family.  The \textit{$L$-functions Ratios Conjecture} has since been employed in a variety of applications, such as computing $n$-level densities across a family of $L$-functions, mollified moments of $L$-functions, and discrete averages over zeros of the Riemann Zeta function \cite{ConreySnaith}.  The Ratios Conjecture has also been extended to the function field setting \cite{Andrade}.  While constructing a model using the Ratios Conjecture may pose additional technical challenges, the reward is often a more accurate model; RMT heuristics can model assymptotic behavior, but the Ratios Conjecture is expected to hold down to lower order terms.  This has been demonstrated, for example, in the context of one-level density computations, by Fiorilli, Parks and S\"odergren \cite{FiorParkS}.
\\
\\
This paper studies $\textnormal{Var}(\psi_{K,X})$ down to lower-order terms.  Define a new parameter $\lambda$ such that $X^{\lambda} = K$.  We prove the following theorem:

\begin{theorem}\label{trivial regime theorem}
Fix $\lambda > 1$.  Then
\begin{equation}
\frac {\textnormal{Var}(\psi_{K,X})}{C_{f}X^{1-\lambda}} = C_{\Phi} \log X+C'_{\Phi}+\pi^2\tilde{\Phi}\left(\frac 1 2\right)^2+o\left(1\right),
\end{equation}
where

\begin{align}
\begin{split}\label{C constant 1}
C_{f} &:=\frac{1}{ 4\pi^2}\int_{\R}f(y)^{2}dy \hspace{15mm} C'_{\Phi} := 4\pi^2 \cdot \int_{0}^{\infty}\log x \cdot \Phi(x)^2 ~dx,
\end{split}
\end{align}
and $C_{\Phi}$ is as in $\textnormal{(\ref{mean value constants})}$. Under GRH, the error term can be improved to $O_{\Phi}\left(X^{-\epsilon}\right)$ for some $\epsilon>0$ $($depending on $\lambda)$.
\end{theorem}

The proof of Theorem \ref{trivial regime theorem} is given in Section 2, and is obtained by classical methods.  For $\lambda < 1$ the computation is more difficult, and we use the Ratios Conjecture to suggest the following.
\begin{conjecture}\label{conj} Fix $0 < \lambda <1 $.  We have
\begin{equation}
\frac {\textnormal{Var}(\psi_{K,X})}{C_{f}X^{1-\lambda}} = 
\left\{
\begin{array}{l l}
C_{\Phi} \log X+\Delta_{\Phi}+O_{\Phi}\left(X^{-\epsilon}\right) & \textnormal{ if } \frac 1 2 <\lambda < 1\\
C_{\Phi}\left(2 \lambda \log X\right) -K_{\Phi}+ O_{\Phi}\left(X^{-\epsilon}\right) & \textnormal{ if } \lambda < \frac 1 2,
\end{array} \right.
\end{equation}
where
\begin{equation}
\Delta_{\Phi}:=C'_{\Phi}- \pi^2\tilde{\Phi}\left(\frac 1 2\right)^{2},
\end{equation}
and
\begin{equation}
K_{\Phi}:= C_{\Phi,\zeta}-C_{\Phi,L}-A_{\Phi}'+2\pi^2\tilde{\Phi}\left(\frac 1 2\right)^{2}+C_{\Phi}\left(\log \left(\frac {\pi^2}{4}\right)+2\right),
\end{equation}
for some constant $\epsilon >0$ $($depending on $\lambda)$.  Here $C_{\Phi,\zeta}$, $C_{\Phi,L}$, and $A_{\Phi}'$, are as in $(\ref{zeta constant})$, $(\ref{L constant})$, and $(\ref{A constant})$, respectively.
\end{conjecture} 

\begin{figure}[h]
\includegraphics[width=8cm, height=6cm]{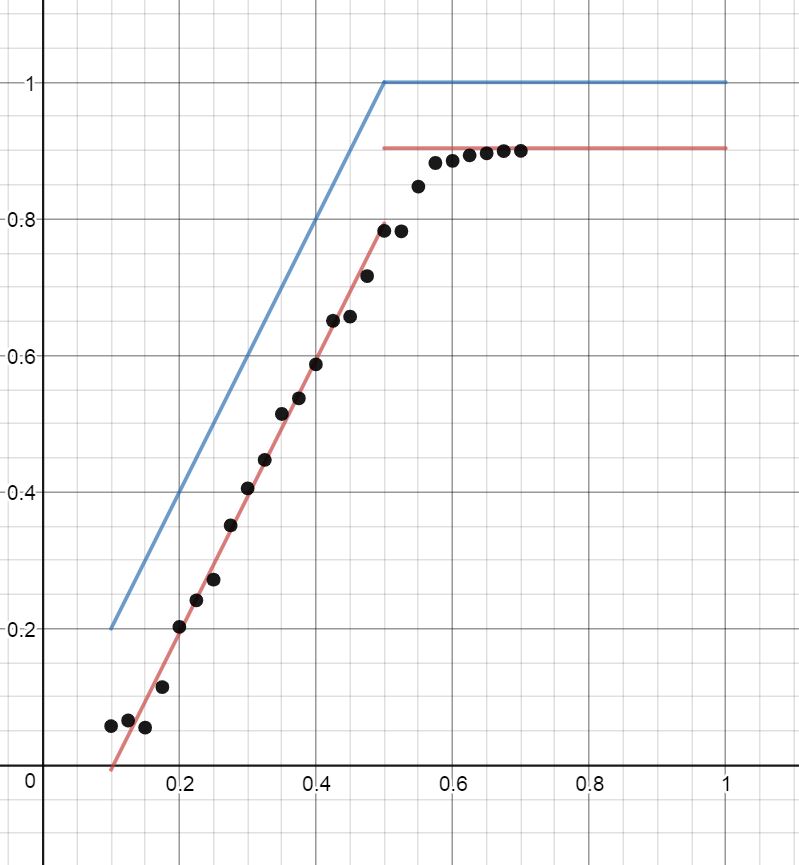} \label{abillionprimes}
\caption{ A plot of the ratio $\textnormal{Var}(\psi_{K,X})/(\langle \psi_{K,X} \rangle \log X)$ versus $\lambda = \log K / \log X$, for $X \approx 10^9$ with test functions $\Phi = 1_{(0,1]}$ and $f = 1_{[-\frac{1}{2},\frac{1}{2}]}$.  The red line is the prediction given by Conjecture \ref{conj}, while the blue line is the RMT Conjecture of (\ref{RMT Conjecture}).}
\centering
\end{figure}

Conjecture \ref{conj} provides a refined conjecture for $\textnormal{Var}(\psi_{K,X})$ with a power saving error term (away from the bifurcation points).  It moreover recovers the asymptotic prediction given by (\ref{RMT Conjecture}), which was initially obtained by completely different methods.  Numerical data for $\textnormal{Var}(\psi_{K,X})$ is provided in Figure ~1.\\
\\  A saturation effect similar to the one above was previously observed by Bui, Keating, and Smith \cite{BuiKeatingSmith}, when computing the variance of sums in short intervals of coefficients of a fixed $L$-function of high degree.  There, too, the contribution from lower order terms must be taken into account in order to obtain good agreement with the numerical data.\\
\\
A proof of Theorem \ref{trivial regime theorem} is provided in Section 2 below.  When $\lambda >1$ the main contribution to the variance is given by the diagonal terms, which we directly compute by separately considering the weighted contribution of split primes (Lemma \ref{Split Primes}) and inert primes (Lemma \ref{Inert Primes}).  When $0 < \lambda < 1$ we may no longer trivially bound the off-diagonal contribution, and so we instead shift focus to the study of a relevant family of Hecke $L$-functions.  In Section 3 we compute the ratios recipe for this family of $L$-functions, and in Section 4 we apply several necessary simplifications.  Section 5 then relates the output of this recipe to $\textnormal{Var}(\psi_{K,X})$, resulting in Conjecture \ref{full conjecture}, which expresses $\textnormal{Var}(\psi_{K,X})$ in terms of four double contour integrals.  Section 6 is dedicated to preliminary technical lemmas, and the double integrals are then computed in Sections 7$ - $9.  One finds that the main contributions to $\textnormal{Var}(\psi_{K,X})$ come from second-order poles, while first-order poles contribute a correction factor smaller than the main term by a factor of $\log X$.\\
\\
The Ratios Conjectures moreover suggests an enlightening way to group terms.  The first integral, which corresponds to taking the first piece of each approximate functional equation in the ratios recipe, corresponds to the contribution of the diagonal terms, computed in Theorem \ref{trivial regime theorem}.  In particular, we note that its contribution to $\textnormal{Var}(\psi_{K,X})$ is independent of the value of $\lambda$ (Lemma \ref{1st total}).  In contrast, the contribution emerging from the second and third integrals depends on the value of $\lambda$ (Lemma \ref{second total}).  This accounts for the emergence of two bifurcation points in the lower order terms: one at $\lambda = 1/2$ and another at $\lambda = 1$.  The fourth integral, corresponding to taking the second piece of each approximate functional equation in the ratios recipe, only makes a significantly contribution to $\textnormal{Var}(\psi_{K,X})$ when $\lambda < 1/2$ (Lemma \ref{fourth total}).  This accounts for the bifurcation point in the main term, previously detected by the RMT model, as well as for the contribution of a complicated lower-order term, which appears to nicely fit the numerical data. \\
\\
\textbf{Acknowledgments:} This work emerged from a summer project developed and guided by E. Waxman, as part of the 2017 SMALL  Undergraduate Research Project at Williams College.  We thank Zeev Rudnick for advice, and for suggesting this problem, as well as Bingrong Huang and J. P. Keating for helpful discussions.  The summer research was supported by  NSF Grant DMS1659037.  Chen was moreover supported by Princeton University, and Miller was supported by NSF Grant DMS1561945.  Waxman was supported by the European Research Council under the European Union's Seventh Framework Programme (FP7/2007-2013) / ERC grant agreement no 320755., as well as by the Czech Science Foundation GA\v CR, grant 17-04703Y.
\section{Proof of Theorem \ref{trivial regime theorem}}
Recall that $X^{\lambda} = K$.  To compute $\textnormal{Var}(\psi_{K,X})$ in the regime $\lambda > 1$, it suffices to calculate the second moment, defined as
\begin{align}
\begin{split}
\Omega_{K,X} &:= \frac{1}{\pi/2}\int_{0}^{\frac \pi 2}\bigg|\psi_{K,X}(\theta)\bigg|^{2}d\theta\\
&= \frac{2}{\pi} \sum_{\substack{\A, \B \subset \Z[i]\\ }} \Phi \left(\frac{N(\A)}{X}\right)\Phi \left(\frac{N(\B)}{X}\right) \Lambda(\A)\Lambda(\B) \int_{0}^{\frac \pi 2} F_K(\theta_{\A}  -\theta)F_K(\theta_{\B}  -\theta)d\theta.\\
\end{split}
\end{align}
Indeed, note that as in Lemma 3.1 of \cite{RudWax}, 
\begin{equation}
\langle\psi_{K,X}\rangle \sim \frac{X}{K}\int_{\R}f(x)dx\int_{0}^{\infty}\Phi(u)du\\
 = O\left(\frac{X}{K}\right),
\end{equation}
so that for $\lambda > 1$,
\begin{align} \label{var vs. second moment}
\begin{split}
\textnormal{Var}(\psi_{K,X}) &= \Omega_{K,X} - \langle\psi_{K,X}\rangle^{2}\\
&= \Omega_{K,X} + O\left(X^{1-\epsilon}\right),
\end{split}
\end{align}
where $\epsilon = 2\lambda-1$.\\
\\
Suppose $\A \neq \B$, and that at least one of $\theta_{\A}, \theta_{\B} \neq 0$.  Then by Lemma 2.1 in \cite{RudWax},
\begin{equation}
|\theta_{\A}-\theta_{\B}| \geq \frac{1}{X} \gg \frac{1}{K}.
\end{equation}
Moreover, in order for the integral
\begin{equation}
\int_{0}^{\pi/2} F_K(\theta_{\A}  -\theta)F_K(\theta_{\B}  -\theta)d\theta
\end{equation}
to be nonzero, we require that $\theta_{\A} - \theta_{\B}< \frac{\pi}{2 K}$.  Since $X = o(K)$, such off-diagonal terms contribute nothing, and the contribution thus only comes from terms for which $\theta_{\A} = \theta_{\B}$.  We therefore may write
\begin{align} 
\begin{split}
\Omega_{K,X} &= \frac{2}{\pi}\sum_{\substack{\A \subset \Z[i] \\ \theta_{\A}\neq 0}} \Phi \left(\frac{N(\A)}{X}\right)^2 \Lambda^2(\A)\int_{0}^{\frac \pi 2} F_K(\theta)^2d\theta\\
&\phantom{=}+\frac{2}{\pi}\bigg |\sum_{\substack{\A \subset \Z[i] \\ \theta_{\A} = 0}} \Phi \left(\frac{N(\A)}{X}\right) \Lambda(\A)\bigg |^2 \int_{0}^{\frac \pi 2} F_K(\theta)^2d\theta.
\end{split}
\end{align}
By Parseval's theorem we have that for sufficiently large $K$,
\begin{align}\label{fourier simplification}
\begin{split}
\frac{2}{\pi}\int_{0}^{\frac \pi 2} \vert F_K(\theta)\vert^2d\theta &= \sum_{k \in \Z}\vert \widehat{F}_K(k)\vert^2d\theta= \frac{1}{K^{2}}\sum_{k \in \Z}\widehat{f}\left(\frac k K \right)^2 = 4\pi ^2 \frac {C_{f}}{K},
\end{split}
\end{align}
and therefore
\begin{align} \label{second moment}
\begin{split}
\Omega_{K,X} &= 4\pi ^2 \frac {C_{f}}{K}\left(\sum_{\substack{\A \subset \Z[i] \\ \theta_{\A}\neq 0}} \Phi \left(\frac{N(\A)}{X}\right)^2 \Lambda^2(\A)+\bigg |\sum_{\substack{\A \subset \Z[i] \\ \theta_{\A} = 0}} \Phi \left(\frac{N(\A)}{X}\right) \Lambda(\A)\bigg |^2\right).
\end{split}
\end{align}
Theorem \ref{trivial regime theorem} then follows from (\ref{var vs. second moment}), (\ref{second moment}), 
and the following two lemmas.
\begin{lemma}\label{Split Primes}
We have
\begin{equation}\label{error no grh}
\sum_{\substack{\A \subset \Z[i] \\ \theta_{\A}\neq 0}} \Phi \left(\frac{N(\A)}{X}\right)^2 \Lambda^2(\A) = \frac{1}{4\pi^2}\bigg(C_{\Phi}X\cdot \log X-X C'_{\Phi}\bigg)+O_{\Phi}\left(Xe^{-c \cdot \sqrt{\log X}}\right),
\end{equation}
while under GRH, the error term has a power saving, say, to $O_{\Phi}\left(X^{2/3}\right)$.
\end{lemma}
\begin{lemma}\label{Inert Primes}
Unconditionally we have that
\begin{equation}
\bigg|\sum_{\substack{\A \subset \Z[i] \\ \theta_{\A} = 0}} \Lambda(\A) \Phi\left(\frac{N(\A)}{X}\right)\bigg|^{2} = \frac{X}{4} \left(\tilde{\Phi}\left(\frac 1 2\right)\right)^{2}+O_{\Phi}\left(Xe^{-c \cdot \sqrt{\log X}}\right),
\end{equation}
while, again, under GRH, the error term has a power saving.
\end{lemma}
\textbf{Proof of Lemma \ref{Split Primes}:}
\begin{proof}
Consider the quantity
\begin{align}
\begin{split}
\sum_{\substack{\A \subset \Z[i] \\ \theta_{\A}\neq 0}}& \Phi \left(\frac{N(\A)}{X}\right)^2 \Lambda^2(\A) = \sum_{\p|p\equiv 1(4)}\sum_{n=1}^{\infty} \Phi \left(\frac{N(\p^{n})}{X}\right)^2 \Lambda^2(\p)+\sum_{m=0}^{\infty}\Phi \left(\frac{2^{2m+1}}{X}\right)^2 (\log 2)^2\\
&=\sum_{p \equiv 1(4)}2\cdot \Phi \left(\frac{p}{X}\right)^2 (\log p)^2+\sum_{\p|p\equiv 1(4)}\sum_{n=2}^{\infty} \Phi \left(\frac{N(\p^{n})}{X}\right)^2 \Lambda^2(\p)+O_{\Phi}\left(\log X\right),
\end{split}
\end{align}
where we note that since $\Phi$ is compactly supported, the sum on the far right has at most $O_{\Phi}\left(\log X\right)$ terms.  Moreover,
\begin{align}
\sum_{\p|p\equiv 1(4)}\Phi \left(\frac{N(\p^{n})}{X}\right)^2 \Lambda^2(\p)&\ll X^{\frac 1 n+\epsilon}
\end{align}
since the sum has at most $O_{\Phi}(X^{1/n})$ terms.  It follows that
\begin{equation}
\sum_{\p|p\equiv 1(4)}\sum_{n=2}^{\infty} \Phi \left(\frac{N(\p^{n})}{X}\right)^2 \Lambda^2(\p) \ll \sum_{n=2}^{\log X} X^{\frac{1}{n}+\epsilon}= O_{\Phi}\left(X^{\frac{2}{3}}\right),
\end{equation}
and therefore
\begin{equation}\label{split simplification}
\sum_{\substack{\A \subset \Z[i] \\ \theta_{\A}\neq 0}} \Phi \left(\frac{N(\A)}{X}\right)^2 \Lambda^2(\A) = \sum_{p \equiv 1(4)}2\cdot \Phi \left(\frac{p}{X}\right)^2 (\log p)^2+O_{\Phi}\left(X^{\frac{2}{3}}\right).
\end{equation}
Upon setting
\begin{equation}
f(t) := \log t \cdot \Phi\left(\frac{t}{X}\right)^2
\end{equation}
and
\begin{equation}
a_{p} := \left\{
\begin{array}{l l}
2\cdot \log p & \text{ if } p \equiv 1(4) \\
0 & \text{ otherwise,}
\end{array} \right.
\end{equation}
it follows from Abel's Summation Formula and the Prime Number Theorem that
\begin{align}\label{abel split}
\begin{split}
 \sum_{p \equiv 1(4)}2\cdot \Phi \left(\frac{p}{X}\right)^2 (\log p)^2
&= \int_{1}^{\infty}\log t \cdot \Phi\left(\frac{t}{X}\right)^2 dt+O\left(\int_{1}^{\infty}t^{\frac 1 2 +\epsilon}\cdot f'(t)dt\right).
\end{split}
\end{align}
where the error term assumes RH.  Applying the change of variables $u :=t/X$, we then obtain that for sufficiently large $X$,
\begin{align}\label{abel main}
\begin{split}
\int_{1}^{\infty}\log t \cdot \Phi\left(\frac{t}{X}\right)^2 dt &= X\cdot \log X \int_{0}^{\infty} \Phi\left(u\right)^2 du+X\cdot \int_{0}^{\infty}\log u\cdot  \Phi\left(u\right)^2 du\\
&= \frac{1}{4\pi^2}\bigg(X\cdot \log X C_{\Phi}-X C'_{\Phi}\bigg).
\end{split}
\end{align}
Under RH, the error term is then given as
\begin{align}\label{abel error}
\begin{split}
\int_{1}^{\infty}t^{\frac 1 2+\epsilon}\cdot f'(t)dt & \ll \int_{1}^{\infty}t^{-\frac 1 3} \cdot \Phi\left(\frac{t}{X}\right)^2 dt\ll_{\Phi} X^{\frac 2 3},
\end{split}
\end{align}
while unconditionally it is as in (\ref{error no grh}).  Combining the results of (\ref{split simplification}), (\ref{abel split}), (\ref{abel main}), and (\ref{abel error}), we then obtain Lemma \ref{Split Primes}.
\end{proof}
\textbf{Proof of Lemma \ref{Inert Primes}:}
\begin{proof}
Next, we consider the quantity
\begin{align}
\begin{split}
\sum_{\substack{\A \subset \Z[i] \\ \theta_{\A} = 0}} \Phi \left(\frac{N(\A)}{X}\right) \Lambda(\A) &= 2\sum_{p \equiv 3(4)}\sum_{j=1}^{\infty} \Phi \left(\frac{p^{2j}}{X}\right) \log p+\sum_{m=1}^{\infty} \Phi \left(\frac{2^{2m}}{X}\right) \log 2\\
&= 2\sum_{p \equiv 3(4)}\sum_{j=1}^{\infty} \Phi \left(\frac{p^{2j}}{X}\right) \log p+O_{\Phi}\left(\log X\right).
\end{split}
\end{align}

Since
\begin{equation}
\sum_{p \equiv 3(4)}\Phi \left(\frac{p^{2j}}{X}\right) \log p \ll_{\Phi} X^{\frac{1}{2j}+\epsilon},
\end{equation}
we have that
\begin{equation}
\sum_{p \equiv 3(4)}\sum_{j=2}^{\infty} \Phi \left(\frac{p^{2j}}{X}\right) \log p \ll_{\Phi} (\log X)\cdot X^{\frac{1}{4}+\epsilon}=O_{\Phi}\left( X^{\frac 1 3}\right),
\end{equation}
and therefore
\begin{equation}
\sum_{\substack{\A \subset \Z[i] \\ \theta_{\A} = 0}} \Phi \left(\frac{N(\A)}{X}\right) \Lambda(\A) = 2 \sum_{p \equiv 3(4)}\Phi \left(\frac{p^2}{X}\right) \Lambda (p)+O_{\Phi}\left(X^{\frac{1}{3}}\right).
\end{equation}
Moreover, since
\begin{align}
\begin{split}
\sum_{n \equiv 3(4)}\Phi \left(\frac{n^2}{X}\right) \Lambda (n) &=\sum_{p \equiv 3(4)}\Phi \left(\frac{p^2}{X}\right) \Lambda (p)+ \sum_{p \equiv 3(4)}\sum_{\substack{j=3 \\ \textnormal{odd}}}^{\infty}\Phi \left(\frac{p^{2j}}{X}\right) \Lambda (p)\\
&=\sum_{p \equiv 3(4)}\Phi \left(\frac{p^2}{X}\right) \Lambda (p)+ O_{\Phi}\left(X^{\frac{1}{3}}\right),
\end{split}
\end{align}
we obtain
\begin{equation}
\sum_{\substack{\A \subset \Z[i] \\ \theta_{\A} = 0}} \Phi \left(\frac{N(\A)}{X}\right) \Lambda(\A) = 2\sum_{n \equiv 3(4)}\Phi \left(\frac{n^2}{X}\right) \Lambda (n)+O_{\Phi}\left(X^{\frac{1}{3}}\right).
\end{equation}
By the Mellin inversion theorem, we find that
\begin{align}
\begin{split}
\sum_{n \equiv 3(4)}\Phi \left(\frac{n^2}{X}\right) \Lambda (n) & =\sum_{n \equiv 3(4)}  \Lambda (n) \frac{1}{2\pi i}\int_{(2)}\tilde{\Phi}(s)\left(\frac{n^2}{X}\right)^{-s}ds \\
& = \frac{1}{2\pi i}\int_{(2)}\tilde{\Phi}(s)\sum_{n \equiv 3(4)}\frac{\Lambda (n)}{n^{2s}}X^{s}ds.
\end{split}
\end{align}
Let $\chi_{0} \in \left(\Z/4\Z\right)^{\times}$ denote the principal character, and $\chi_{1} \in \left(\Z/4\Z\right)^{\times}$ denote the non-principal character, with corresponding $L$-functions given by $L(s, \chi_{0})$ and $L(s, \chi_{1})$, respectively.  Upon noting that
\begin{equation}
\chi_{0}(n)-\chi_{1}(n)=
\left\{
\begin{array}{l l}
2 & \text{ if } n = 3 \text{ mod }4 \\
0 & \text{otherwise}, \\
\end{array} \right.
\end{equation}
we obtain
\begin{align}
\begin{split}
\frac{L'}{L}(2s,\chi_1)-\frac{L'}{L}(2s,\chi_0) &=\sum_{n=1}^{\infty}\frac{\Lambda(n)(\chi_{0}(n)-\chi_{1}(n))}{n^{2s}}\\
&=2\sum_{n\equiv 3(4)}^{\infty}\frac{\Lambda(n)}{n^{2s}}.
\end{split}
\end{align}
It follows that
\begin{align}\label{integral2s}
\begin{split}
2\sum_{n \equiv 3(4)}\Phi \left(\frac{n^2}{X}\right) \Lambda (n) &= \frac{1}{2\pi i}\int_{(2)}\left(\frac{L'}{L}(2s,\chi_{1})-\frac{L'}{L}(2s,\chi_{0})\right)\tilde{\Phi}(s)X^s ds\\
&= \frac{1}{4\pi i}\int_{(4)}\left(\frac{L'}{L}(s,\chi_{1})-\frac{L'}{L}(s,\chi_0)\right)\tilde{\Phi}\left(\frac s 2\right)X^\frac{s}{2} ds.
\end{split}
\end{align}
Moreover, we compute
\begin{equation}
\frac{L'}{L}(s,\chi_{0}) = -\frac{1}{s-1}+\gamma_{0}+\log 2 +O(s-1),
\end{equation}
where $\gamma_{0}$ is the Euler-Mascheroni constant, while $L'/L(s,\chi_1)$ is holomorphic about $s=1$.  Shifting integrals, we pick up a pole at $s = 1$ and find that 
\begin{equation}
\sum_{\substack{\A \subset \Z[i] \\ \theta_{\A} = 0}} \Lambda(\A) \Phi\left(\frac{N(\A)}{X}\right) = \frac{1}{2} X^{\frac{1}{2}}\tilde{\Phi}\left(\frac 1 2\right)+O_{\Phi}\left(\sqrt{X}e^{-c \cdot \sqrt{\log X}}\right)
\end{equation}
for some $c > 0$.  Squaring this then yields
\begin{equation}
\left|\sum_{\substack{\A \subset \Z[i] \\ \theta_{\A} = 0}} \Lambda(\A) \Phi\left(\frac{N(\A)}{X}\right)\right|^{2} = \frac{X}{4} \left(\tilde{\Phi}\left(\frac 1 2\right)\right)^{2}+O_{\Phi}\left(Xe^{-c \cdot \sqrt{\log X}}\right).
\end{equation}
As above, we note that under the assumption of GRH the error term can be improved to have a power-saving.
\end{proof}
\section{Implementing the Ratios Conjecture}
Throughout this section, and the remainder of the paper, we will assume GRH.  
\subsection{The Recipe}
The $L$-Functions Ratios Conjecture described in \cite{Conrey}, provides a procedure for computing an average of $L$-function ratios over a designated family.  Let $\mathcal{L}(s,f)$ be an $L$-function, and $\mathcal{F} =\{f\}$ a family of characters with conductors $c(f)$, as defined in section 3 of \cite{ConreyII}.  $\mathcal{L}(s,f)$ has an approximate functional equation given by
\begin{equation}\label{approx equation}
\mathcal{L}(s,f) = \sum_{n<x}\frac{A_{n}(f)}{n^s}+\epsilon(f,s)\sum_{m <y} \frac{\overline{A_{m}(f)}}{m^{1-s}}+\textnormal{remainder}.
\end{equation}
Moreover, one may write
\begin{equation} \label{denominator}
\frac{1}{\mathcal{L}(s,f)} = \sum_{n=1}^{\infty}\frac{\mu_{f}(n)}{n^s},
\end{equation}
where the series converges absolutely for Re$(s)>1$.
To conjecture an asymptotic formula for the average
\begin{equation}
\sum_{f \in \mathcal{F}}\frac{\mathcal{L}(\frac{1}{2} +\alpha,f)\mathcal{L}(\frac{1}{2} + \beta,f)}{\mathcal{L}(\frac{1}{2}+ \gamma,f)\mathcal{L}(\frac{1}{2} + \delta,f)},
\end{equation}
the \textit{Ratios Conjecture} suggests the following recipe.\\
\\
\textbf{Step One}: Start with
\begin{equation}
\frac{\mathcal{L}(\frac{1}{2}+ \alpha,f)\mathcal{L}(\frac{1}{2} + \beta,f)}{\mathcal{L}(\frac{1}{2}+ \gamma,f)\mathcal{L}(\frac{1}{2} + \delta,f)}.
\end{equation}
Replace each $L$-function in the numerator with the two terms from its approximate functional equation, ignore the remainder terms and allow each of the four resulting sums to extend to infinity.  Replace each $L$-function in the denominator by its series (\ref{denominator}).  Multiply out the resulting expression to obtain 4 terms.  Write these terms as 
\begin{equation}(\text{product of }\epsilon(f,s) \text{ factors})\sum_{n_{1},\dots, n_{4}}(\text{summand}).
\end{equation}
\textbf{Step Two}: Replace each product of $\epsilon(f,s)$ factors by its expected value when averaged over the family.\\
\\
\textbf{Step Three}: Replace each summand by its expected value when averaged over the family.\\
\\
\textbf{Step Four}: Call the total $M_{f}:=M_{f}(\alpha,\beta,\gamma,\delta)$, and let $F = |\mathcal{F}| $.  Then for 
\begin{equation}\label{ratios domain}
-\frac{1}{4}<\textnormal{Re}(\alpha),\textnormal{Re}(\beta)< \frac{1}{4}, \hspace{10mm}
\frac{1}{\log F}\ll \textnormal{Re}(\gamma), \textnormal{Re}(\delta)<\frac{1}{4},
\end{equation}
and
\begin{equation}\label{ratios domain 2}
\textnormal{Im}(\alpha),\textnormal{Im}(\beta),\textnormal{Im}(\gamma),\textnormal{Im}(\delta) \ll_{\epsilon}F^{1-\epsilon},
\end{equation}
the conjecture is that
\begin{equation}
\sum_{f \in \mathcal{F}}\frac{\mathcal{L}(\frac{1}{2} + \alpha,f)\mathcal{L}(\frac{1}{2} + \beta,f)}{\mathcal{L}(\frac{1}{2} + \gamma,f)\mathcal{L}(\frac{1}{2} + \delta,f)}g(c(f))=\sum_{f \in \mathcal{F}}M_{f}\left(1+O\left(e^{(-\frac{1}{2}+\epsilon)c(f)}\right)\right)g(c(f))
\end{equation}
for all $\epsilon > 0$, where $g$ is a suitable weight function.
\subsection{Hecke $L$-functions}
We are interested in applying the ratios recipe to the following family of $L$-functions.  Consider the Hecke character

\begin{equation}
\Xi_{k}(\A):= \left(\alpha/ \overline{\alpha}\right)^{2k} = e^{i 4k \theta_{\A}}, \hspace{5mm} k \in \Z,
\end{equation}
which provides a well-defined function on the ideals of $\Z[i]$.  To each such character we may associate an $L$-function
\begin{align}
L_{k}(s) &:= \ \sum_{\substack{\A \subseteq \Z[i]\\ \A \neq 0}}\frac{\Xi_{k}(\A)}{N(\A)^{s}}=\prod_{\p \textnormal{ prime}}\left(1-\frac{\Xi_{k}(\p)}{N(\p)^s}\right)^{-1}, \hspace{5mm} \textnormal{Re}(s)>1.
\end{align}
Note that $L_{k}(s) = L_{-k}(s)$, and that

\begin{equation}\label{cmplxconj}
\overline{\frac{L'_{k}}{L_{k}}(s)}\ = \ -\sum_{\A \neq 0} \overline{\frac{\Lambda(\A) \Xi_{k}(\A)}{\overline{N(\A)^{s}}}}\ = \ - \sum_{\A \neq 0} \frac{\Lambda(\A)\overline{\Xi_{k}(\A)}}{N(\A)^{\overline{s}}}\ = \ \frac{L'_{-k}}{L_{-k}}(\overline{s})= \frac{L'_{k}}{L_{k}}(\overline{s}).
\end{equation}
Moreover, when $k\neq 0$, then $L_{k}(s)$ has an analytic continuation to the entire complex plane, and satisfies the \textit{functional equation}
\begin{equation}
\xi_{k}(s):=\pi^{-(s+|2k|)}\cdot \Gamma(s+|2k|)\cdot L_{k}(s)=\xi_{k}(1-s).
\end{equation}

\subsection{Step One: Approximate Function Equation}
We seek to apply the above procedure to compute the average

\begin{equation}
\sum_{k \neq 0}\bigg | \widehat{f}\left( \frac{k}{K}\right) \bigg |^{2}\frac{L_{k}(\frac 1 2+\alpha)L_{k}(\frac 1 2+\beta)}{L_{k}(\frac 1 2+\gamma)L_{k}(\frac 1 2+\delta)}
\end{equation}
for specified values of $\alpha, \beta,\gamma, \delta$.  For this particular family of $L$-functions, we have
\begin{equation}
\epsilon(f,s) := \frac{L_{k}(s)}{L_{k}(1-s)}=\pi^{2s-1}\cdot \frac{\Gamma(1-s+|2k|)}{\Gamma(s+|2k|)},
\end{equation}
and
\begin{align}
A_{k}(n) &= \sum_{\substack{N(\A) = n}}\Xi_{k}(\A),
\end{align}
which is a multiplicative function defined explicitly on prime powers by
\begin{equation}\label{coefficients}
A_{k}(p^{l}) = \left\{
\begin{array}{l l}
\sum_{j=-l/2}^{l/2}e^{2j4ki\theta_{p}} & \text{ if } p \equiv 1(4), l \text{ even}\\
\sum_{j=-(l+1)/2}^{(l-1)/2}e^{(2j+1)4ki\theta_{p}} & \text{ if } p \equiv 1(4), l \text{ odd}\\
0 & \text{ if } p \equiv 3(4), l \textnormal{ odd }\\
1 & \text{ if } p \equiv 3(4), l \textnormal{ even } \\
(-1)^{lk} & \text{ if } p = 2,
\end{array} \right.
\end{equation}
where, for prime $p \equiv 1 (4)$, we define $\theta_{p} := \theta_{\p}$, where $\p \subset \Z[i]$ is a prime ideal lying above $p$.  Note, moreover, that the above formula is independent of our specific choice of $\p$.\\
\\
 As per the recipe, we ignore the remainder term and allow both terms in the approximate functional equation to be summed to infinity.  This allows us to write
\begin{equation}
L_{k}(s) \approx \sum_{n}\frac{A_{k}(n)}{n^s}+\pi^{2s-1}\cdot \frac{\Gamma(1-s+|2k|)}{\Gamma(s+|2k|)}\sum_{m} \frac{A_{k}(m)}{m^{1-s}},
\end{equation}
upon noting that $\overline{A_{k}(n)}=A_{k}(n)$ for all $A_{k}(n)$.\\
\\
To compute the inverse coefficients, write

\begin{align}
\begin{split}
\frac{1}{L_k(s)}&=\prod_{\p} \left(1-\frac{e^{4ki \theta_{\p}}}{N(\p)^s}\right)\\
&=\left(1-\frac{(-1)^k}{2^{s}}\right)\prod_{p \equiv 1(4)} \left(1-\frac{(e^{4ki\theta_p}+e^{-4ki\theta_p})}{p^{s}}+\frac{1} {p^{2s}}\right)\prod_{p\equiv 3(4)} \left(1-\frac 1 {p^{2s}}\right)
\\
&=\left(1-\frac{A_{k}(2)}{2^{s}}\right)\prod_{p \equiv 1(4)} \left(1-\frac{A_{k}(p)}{p^{s}}+\frac{1}{p^{2s}}\right)\prod_{p\equiv 3(4)} \left(1-\frac{A_{k}(p)}{p^{s}}-\frac{A_{k}(p^2)}{p^{2s}}\right).
\end{split}
\end{align}
We then obtain
\begin{equation}\label{inverse series}
\frac{1}{L_k(s)}=\sum\limits_{h} \frac{\mu_{k}(h)}{h^s},
\end{equation}
where

\begin{align}\label{inverse coefficients}
\mu_k(p^h):=
\begin{cases}
1 & h=0\\
-A_{k}(p) & h=1
\\
-1 & h=2,p\equiv 3(4)
\\
1 & h=2, p\equiv 1(4)
\\
0 & \textnormal{otherwise}.
\end{cases}
\end{align}
Multiplying out the resulting expression gives
\begin{align}
\begin{split}
&\left(\sum\limits_{h=0}^{\infty} \frac{\mu_{k}(h)}{h^{\frac 1 2+\gamma}}\right)\left(\sum\limits_{l=0}^{\infty} \frac{\mu_{k}(l)}{l^{\frac 1 2+\delta}}\right)\times \left(\sum_{n=0}^{\infty}\frac{A_{k}(n)}{n^{\frac 1 2+\alpha}}+\pi^{2\alpha}\cdot \frac{\Gamma(\frac 1 2-\alpha+|2k|)}{\Gamma(\frac 1 2+\alpha+|2k|)}\sum_{n =0}^{\infty} \frac{A_{k}(n)}{n^{\frac 1 2-\alpha}}\right) \\
&\phantom{=}\times \left(\sum_{m=0}^{\infty}\frac{A_{k}(m)}{m^{\frac 1 2+\beta}}+\pi^{2\beta}\cdot \frac{\Gamma(\frac 1 2-\beta+|2k|)}{\Gamma(\frac 1 2+\beta+|2k|)}\sum_{m =0}^{\infty} \frac{A_{k}(m)}{m^{\frac 1 2-\beta}}\right)
\end{split}
\end{align}
\begin{align}
\begin{split}
&=\prod_{\substack{ p  }}\left(\sum_{m,n,h,l}\frac{\mu_{k}(p^h)\mu_{k}(p^l)A_{k}(p^{n})A_{k}(p^{m})}{p^{h(\frac 1 2+\gamma)+l(\frac 1 2+\delta)+n(\frac 1 2+\alpha)+m(\frac 1 2+\beta)}}\right)\\
&\phantom{=}+\pi^{2\alpha}\cdot \frac{\Gamma(\frac 1 2-\alpha+|2k|)}{\Gamma(\frac 1 2+\alpha+|2k|)}\prod_{\substack{ p  }}\left(\sum_{m,n,h,l}\frac{\mu_{k}(p^h)\mu_{k}(p^l)A_{k}(p^{n})A_{k}(p^{m})}{p^{h(\frac 1 2+\gamma)+l(\frac 1 2+\delta)+n(\frac 1 2-\alpha)+m(\frac 1 2+\beta)}}\right)\\
&\phantom{=}+\pi^{2\beta}\cdot \frac{\Gamma(\frac 1 2-\beta+|2k|)}{\Gamma(\frac 1 2+\beta+|2k|)}\prod_{\substack{ p  }}\left(\sum_{m,n,h,l}\frac{\mu_{k}(p^h)\mu_{k}(p^l)A_{k}(p^{n})A_{k}(p^{m})}{p^{h(\frac 1 2+\gamma)+l(\frac 1 2+\delta)+n(\frac 1 2+\alpha)+m(\frac 1 2-\beta)}}\right) \\
&\phantom{=}+\pi^{2(\alpha+\beta)}\cdot \frac{\Gamma(\frac 1 2-\alpha+|2k|)}{\Gamma(\frac 1 2+\alpha+|2k|)}\frac{\Gamma(\frac 1 2-\beta+|2k|)}{\Gamma(\frac 1 2+\beta+|2k|)}\\
&\phantom{=}\times\prod_{\substack{ p  }}\left(\sum_{m,n,h,l}\frac{\mu_{k}(p^h)\mu_{k}(p^l)A_{k}(p^{n})A_{k}(p^{m})}{p^{h(\frac 1 2+\gamma)+l(\frac 1 2+\delta)+n(\frac 1 2-\alpha)+m(\frac 1 2-\beta)}}\right),
\end{split} 
\end{align}
where the above follows upon noting that
\begin{align}
\begin{split}
&\left(\sum_{h=0}^{\infty}\frac{\mu_{k}(h)}{h^{\frac 1 2+\gamma}}\right)\left(\sum_{l=0}^{\infty}\frac{\mu_{k}(l)}{l^{(\frac 1 2+\delta)}}\right)\left(\sum_{n=0}^{\infty}\frac{A_{k}(n)}{n^{\frac 1 2+\alpha}}\right)\left(\sum_{m=0}^{\infty}\frac{A_{k}(m)}{m^{\frac 1 2+\beta}}\right)\\
&=\prod_{p}\left(\sum_{h}\frac{\mu_{k}(p^h)}{p^{h(\frac 1 2+\gamma)}}\right)\left(\sum_{l}\frac{\mu_{k}(p^{l})}{p^{l(\frac 1 2+\alpha)}}\right)\left(\sum_{n}\frac{A_{k}(p^{n})}{p^{n(\frac 1 2+\alpha)}}\right)\left(\sum_{m}\frac{A_{k}(p^{m})}{p^{m(\frac 1 2+\beta)}}\right)\\
& = \prod_{\substack{ p  }}\left(\sum_{m,n,h,l}\frac{\mu_{k}(p^h)\mu_{k}(p^l)A_{k}(p^{n})A_{k}(p^{m})}{p^{h(\frac 1 2+\gamma)+l(\frac 1 2+\delta)+n(\frac 1 2+\alpha)+m(\frac 1 2+\beta)}}\right).
\end{split}
\end{align}
The algorithm now dictates that we compute the $\Gamma$-average
\begin{equation}
\bigg \langle \pi^{2(\alpha+\beta)}\cdot \frac{\Gamma(\frac 1 2-\alpha+|2k|)}{\Gamma(\frac 1 2+\alpha+|2k|)}\frac{\Gamma(\frac 1 2-\beta+|2k|)}{\Gamma(\frac 1 2+\beta+|2k|)}\bigg \rangle_{K},
\end{equation}
as well as an average for the quantity coming from the first piece of each functional equation, namely 
\begin{equation}
\bigg \langle \prod_{\substack{ p  }}\left(\sum_{m,n,h,l}\frac{\mu_{k}(p^h)\mu_{k}(p^l)A_{k}(p^{n})A_{k}(p^{m})}{p^{h(\frac 1 2+\gamma)+l(\frac 1 2+\delta)+n(\frac 1 2+\alpha)+m(\frac 1 2+\beta)}}\right)\bigg \rangle_{K}.
\end{equation}
Here we write $\langle \cdot \rangle_{K}$ to denote the average over all $0< |k| \leq K$.  The average of the remaining three pieces will then follow similarly upon applying the appropriate change of variables. 
\subsection{Step Two: Averaging the Gamma Factors}
The gamma factor averages over the family of Hecke $L$-functions are provided by the following lemma.
\begin{lemma}
Fix $0 < \alpha, \beta < \frac 1 2$.  We find that 
 
\begin{align}\label{single gamma}
\bigg\langle \frac{\Gamma(\frac 1 2-\alpha+|2k|)}{\Gamma(\frac 1 2+\alpha+|2k|)}\bigg\rangle_{K} = \frac{\left(2K\right)^{-2\alpha}}{1-2\alpha} + O\left(K^{-1}\right),
\end{align}
and similarly
\begin{align}\label{double gamma}
\begin{split}
\bigg\langle \frac{\Gamma(\frac 1 2-\alpha+|2k|)}{\Gamma(\frac 1 2+\alpha+|2k|)}\frac{\Gamma(\frac 1 2-\beta+|2k|)}{\Gamma(\frac 1 2+\beta+|2k|)}\bigg\rangle_{K}&=\frac{(2K)^{-2(\alpha+\beta)}}{1-2(\alpha+\beta)}+ O\left(K^{-1}\right).
\end{split}
\end{align}

\end{lemma}
\begin{proof}
A proof of (\ref{single gamma}) is given in \cite{Waxman}, and the proof of (\ref{double gamma}) is essentially identical.  Specifically, one may use Stirling's approximation and Taylor expansion to demonstrate that
\begin{equation}
\frac{\Gamma\left(\frac 1 2+|2k|-\alpha\right)}{\Gamma\left(\frac 1 2+|2k|+\alpha\right)}\frac{\Gamma\left(\frac 1 2+|2k|-\beta\right)}{\Gamma\left(\frac 1 2+|2k|+\beta\right)}= \left(\frac 1 2+|2k|\right)^{-2(\alpha+\beta)}\left(1+O\left(\frac{1}{k}\right)\right),
\end{equation}
and then average over $0<|k| \leq K$ to obtain (\ref{double gamma}).
\end{proof}
\subsection{Step Three: Coefficient Average}
In this section, we seek to compute the coefficient average
\begin{equation}\label{coefficient average}
\bigg \langle \mu_{k}(p^h)\mu_{k}(p^l)A_{k}(p^{n})A_{k}(p^{m})\bigg \rangle_{K}.
\end{equation}
To do so, we must consider several cases depending on the value of $p$ mod 4.  Define
\begin{align}
\delta_p(m,n,h,l)\ :=\ \lim_{K\to \infty}\bigg \langle \mu_{k}(p^h)\mu_{k}(p^l)A_{k}(p^{n})A_{k}(p^{m})\bigg \rangle_{K}
\end{align}
and write

\begin{align}
\delta_p(m,n,h,l)\ :=\ \begin{cases}
\delta_{3(4)}(m,n,h,l) & \text{ when }p\equiv 3(4)\\
\delta_{1(4)}(m,n,h,l) &  \text{ when }p\equiv 1(4)\\
\delta_{2}(m,n,h,l) &  \text{ when }p=2.
\end{cases}
\end{align}

\subsubsection{}\textit{\textbf{p}} $\equiv$ \textbf{1(4):} By (\ref{inverse coefficients}), we may restrict to the case in which $h,l 
\in \{0,1,2\}$. If $h,l \in \{0,2\}$, then $\delta_{1(4)}(m,n,h,l)$ reduces to $\left<A_{k}(p^m)A_{k}(p^n)\right>_K$, where
\begin{equation}
A_{k}(p^{m}) = \left\{
\begin{array}{l l}
\sum_{j=-\frac{m}{2}}^{\frac{m}{2}}e^{2j4ki\theta_p} &  m \text{ even}\\
\sum_{j=-\frac{\left(m+1\right)}{2}}^{\frac{\left(m-1\right)}{2}}e^{\left(2j+1\right)4ki\theta_p} & m \textnormal{ odd}.\\
\end{array} \right.
\end{equation}
Expanding the product $A_{k}(p^m)A_{k}(p^n)$ yields a double sum of points on the unit circle, and averaging over $ k \leq K$ then eliminates, in the limit, any such terms which are not identically equal to 1.  Collecting the significant terms, we find that
\begin{align}
\delta_{1(4)}(m,n,h,l) =  
\begin{cases}
\min{\left\{m,n\right\}}+1& m+n ~ \mathrm{even}\\
0& m+n ~ \mathrm{odd}.
\end{cases}
\end{align}
If either $h=1$ and $l \in \{0,2\}$, or $l=1$ and $h \in \{0,2\}$, then the product
$\mu_{k}(p^h)\mu_{k}(p^l)=-A_{k}(p)=-(e^{4ki\theta_p}+e^{-4ki\theta_p})$, so that (\ref{coefficient average}) reduces to
\begin{equation}
\left<-\left(e^{4ki\theta_p}+e^{-4ki\theta_p}\right)A_{k}(p^m)A_{k}(p^n)\right>_K.
\end{equation} Expanding out this product yields again a sum of points on the unit circle, which upon averaging over $k \leq K$ eliminates, in the limit, any such terms not identically equal to 1. We then obtain
\begin{align}
\delta_{1(4)}(m,n,h,l)=
\begin{cases}
0 & m+n ~ \mathrm{even}\\
-2\left(\min{\left\{m,n\right\}}+1\right) & m+n ~ \mathrm{odd}.
\end{cases}
\end{align}
Finally, suppose $h = l = 1$.  In this case, the product $\mu_{k}(p^h)\mu_{k}(p^l)=A_{k}(p)^2=e^{2\cdot4ki\theta_p}+2+e^{-2\cdot4ki\theta_p}$, so that (\ref{coefficient average}) reduces to 
\begin{equation}
\left<\left(e^{2\cdot 4ki\theta_p}+2+e^{-2\cdot 4ki\theta_p}\right)A_{k}(p^m)A_{k}(p^n)\right>_K.
\end{equation} Collecting significant contributions as before, we conclude that
\begin{align}
\delta_{1(4)}(m,n,h,l)=
\begin{cases}
4n+2 & m=n\\
4\left(\min{\left\{m,n\right\}}+1\right)& m\neq n, m+n ~ \mathrm{even}\\
0 & m+n ~ \mathrm{odd}.
\end{cases}
\end{align}

\subsubsection{}\textit{\textbf{p}} $\equiv$ \textbf{3(4):}
Again we may restrict to the case in which $h,l \in \{0,2\}$.  If $h = l
\in \{0,2\}$, then $\mu_{k}(p^h)\mu_{k}(p^l)=1$, and therefore 
\begin{align}
\delta_{3(4)}(m,n,h,l)&=
\begin{cases}
1 & m, n \text{ are even}
\\
0 & \text{otherwise}.
\end{cases}
\end{align}
Likewise, if $(h,l) = (0,2)$ or $(h,l)=(2,0)$ then $\mu_{k}(p^h)\mu_{k}(p^l)=-1$ and 
\begin{align}
\delta_{3(4)}(m,n,h,l) &=
\begin{cases}
-1 & m, n \text{ are even}
\\
0 & \text{otherwise}.
\end{cases}
\end{align}

\subsubsection{}\textit{\textbf{p}} \textbf{= 2:} When $p=2$, we may restrict to the case in which $h,l \in \{0,1\}$.  If, moreover, $h=l$, then 

\begin{align}
\delta_{2}(m,n,h,l)&=\big \langle(-1)^{(m+n)k}\big \rangle_K =
\begin{cases}
1 & m+n \text{ is even}
\\
0 & \text{otherwise,}
\end{cases}
\end{align}
while if $h \neq l$, 
\begin{align}
\delta_{2}(m,n,h,l)&= -\big \langle(-1)^{(m+n+1)k}\big \rangle_K =
\begin{cases}
-1 & m+n \text{ is odd}
\\
0 & \text{otherwise.}
\end{cases}
\end{align}
\subsubsection{}\textit{\textbf{Summary:}}
Summarizing the above results, we then conclude that
\begin{align}\label{coefficient cases}
\begin{split}
&\delta_{1(4)}(m,n,h,l)\ =\ 
\begin{cases}
\min{\{m,n\}}+1 & m+n\text{ even, } h,l \in \{0,2\}\\
-2(\min{\{m,n\}}+1) & m+n\text{ odd, }(h,l)=(0,1),(1,0),(1,2)\text{ or }(2,1)\\
4n+2 & m=n,\ (h,l) = (1,1)\\
4\left(\min{\left\{m,n\right\}}+1\right) & m\neq n, \hspace{2mm} m+n ~ \mathrm{even,} \hspace{2mm} (h,l) = (1,1)\\
0 & \text{otherwise},
\end{cases}
\\
&\delta_{3(4)}(m,n,h,l)\ =\ 
\begin{cases}
1 & m,n\text{ even, }(h,l)=(0,0)\text{ or }(2,2)\\
-1 & m,n\text{ even, }(h,l)=(0,2)\text{ or }(2,0)\\
0 &\text{otherwise},
\end{cases}\\
&\hspace{4mm}\delta_2(m,n,h,l)\ =\ 
\begin{cases}
1 & m+n\text{ even, }(h,l)=(0,0)\text{ or }(1,1)\\
-1 & m+n\text{ odd, }(h,l)=(0,1)\text{ or }(1,0)\\
0 & \text{otherwise}.
\end{cases}
\end{split}
\end{align}
\subsection{Step Four: Conjecture}
Upon applying the averages, the Ratios Conjecture recipe claims that for $\alpha, \beta,\gamma,\delta$ satisfying the conditions specified in (\ref{ratios domain}), we have

\begin{align}\label{the conjecture M}
\sum_{k \neq 0}\bigg | \widehat{f}\left( \frac{k}{K}\right) \bigg |^{2}\frac{L_{k}(\frac 1 2+\alpha)L_{k}(\frac 1 2+\beta)}{L_{k}(\frac 1 2+\gamma)L_{k}(\frac 1 2+\delta)}=\sum_{k \neq 0}\bigg | \widehat{f} \left( \frac{k}{K}\right) \bigg |^{2} M_{K}(\alpha,\beta,\gamma,\delta)+O\left(K^{\frac{1}{2}+\epsilon}\right),
\end{align}
where
\begin{align}
\begin{split}
&M_{K}(\alpha,\beta,\gamma,\delta):=\prod_{\substack{ p  }}G_{p}(\alpha,\beta,\gamma,\delta)+\frac{\left(\pi/2K\right)^{2\alpha}}{1-2\alpha}\prod_{\substack{ p  }}G_{p}(-\alpha,\beta,\gamma,\delta)\\
&\phantom{=}+\frac{\left(\pi/2K\right)^{2\beta}}{1-2\beta} \prod_{\substack{ p  }}G_{p}(\alpha,-\beta,\gamma,\delta)+\frac{\left(\pi/2K\right)^{2(\alpha+\beta)}}{1-2(\alpha+\beta)}\prod_{\substack{ p  }}G_{p}(-\alpha,-\beta,\gamma,\delta),
\end{split}
\end{align}

and
\begin{equation}
G_{p}(\alpha,\beta,\gamma,\delta):=\sum_{m,n,h,l}\frac{\delta_{p}(m,n,h,l)}{p^{h(\frac 1 2+\gamma)+l(\frac 1 2+\delta)+n(\frac 1 2+\alpha)+m(\frac 1 2+\beta)}}.
\end{equation}

\section{Simplifying the Ratios Conjecture Prediction}

In this section we seek a simplified form of $M_{K}(\alpha,\beta,\gamma,\delta)$.  First, we again consider several separate cases, depending on the value of $p$ mod 4.
\subsection{Pulling out Main Terms}
Suppose $p \equiv 3(4)$. By (\ref{coefficient cases}), we expand each local factor as
\begin{align}
\begin{split}
G_{p}(\alpha,\beta,\gamma,\delta)&=\sum_{\substack{m,n \\ \text{even}}}\frac{\delta_{3(4)}(m,n,0,0)}{p^{n(\frac 1 2+\alpha)+m(\frac 1 2+\beta)}}+\frac{\delta_{3(4)}(m,n,2,2)}{p^{2(\frac 1 2+\gamma)+2(\frac 1 2+\delta)+n(\frac 1 2+\alpha)+m(\frac 1 2+\beta)}}\\
&\phantom{=}+ \frac{\delta_{3(4)}(m,n,0,2)}{p^{2(\frac 1 2+\delta)+n(\frac 1 2+\alpha)+m(\frac 1 2+\beta)}}+\frac{\delta_{3(4)}(m,n,2,0)}{p^{2(\frac 1 2+\gamma)+n(\frac 1 2+\alpha)+m(\frac 1 2+\beta)}}\\
&=\sum_{\substack{m,n \\ \text{even}}}\frac{1}{p^{n(\frac 1 2+\alpha)+m(\frac 1 2+\beta)}}+\frac{1}{p^{(1+2\gamma)+(1+2\delta)+n(\frac 1 2+\alpha)+m(\frac 1 2+\beta)}}\\
&\phantom{=}- \frac{1}{p^{(1+2\delta)+n(\frac 1 2+\alpha)+m(\frac 1 2+\beta)}}-\frac{1}{p^{(1+2\gamma)+n(\frac 1 2+\alpha)+m(\frac 1 2+\beta)}}\\
&=\left(1+\frac{1}{p^{2+2\gamma+2\delta}}-\frac{1}{p^{1+2\delta}}-\frac{1}{p^{1+2\gamma}}\right)\sum_{m,n}\frac{1}{p^{n(1+2\alpha)+m(1+2\beta)}}.
\end{split}
\end{align}
Assuming small positive fixed values of $\text{Re}(\alpha),\text{Re}(\beta),\text{Re}(\gamma),\text{Re}(\delta)$, we factor out all terms which, for fixed $p$, converge substantially slower than $1/p^2$ and note that
\begin{align}
\begin{split}
G_{p}(\alpha,\beta,\gamma,\delta)&=\left(1-\frac{1}{p^{1+2\delta}}-\frac{1}{p^{1+2\gamma}}+O\left(\frac{1}{p^{2}}\right)\right)\left(1+\frac{1}{p^{1+2\alpha}}+\frac{1}{p^{1+2\beta}}+O\left(\frac{1}{p^{2}}\right)\right)\\
&=1-\frac{1}{p^{1+2\delta}}-\frac{1}{p^{1+2\gamma}}+\frac{1}{p^{1+2\alpha}}+\frac{1}{p^{1+2\beta}}+O\left(\frac{1}{p^{2}}\right)\\
&=\left(1-\frac{1}{p^{1+2\alpha}}\right)^{-1}\left(1-\frac{1}{p^{1+2\beta}}\right)^{-1}\left(1-\frac{1}{p^{1+2\gamma}}\right)\left(1-\frac{1}{p^{1+2\delta}}\right)+O\left(\frac{1}{p^{2}}\right).
\end{split}
\end{align}
In fact we write
\[G_{p}(\alpha,\beta,\gamma,\delta)=Y_{p}(\alpha,\beta,\gamma,\delta)\times A_{p}(\alpha,\beta,\gamma,\delta),\]
where
\begin{align}
Y_{p}(\alpha,\beta,\gamma,\delta):=&\frac{\left(1-\frac{1}{p^{1+\alpha+\gamma}}\right)\left(1-\frac{1}{p^{1+\beta+\gamma}}\right)
\left(1-\frac{1}{p^{1+\alpha+\delta}}\right)\left(1-\frac{1}{p^{1+\beta+\delta}}\right)}{\left(1-\frac{1}{p^{1+2\alpha}}\right)\left(1-\frac{1}{p^{1+2\beta}}\right)\left(1-\frac{1}{p^{1+\alpha+\beta}}\right)\left(1-\frac{1}{p^{1+\gamma+\delta}}\right)}\nonumber 
\\
& \times\frac{\left(1+\frac{1}{p^{1+\alpha+\gamma}}\right)\left(1+\frac{1}{p^{1+\beta+\gamma}}\right)
\left(1+\frac{1}{p^{1+\alpha+\delta}}\right)\left(1+\frac{1}{p^{1+\beta+\delta}}\right)}{\left(1+\frac{1}{p^{1+\alpha+\beta}}\right)\left(1+\frac{1}{p^{1+2\gamma}}\right)\left(1+\frac{1}{p^{1+2\delta}}\right)\left(1+\frac{1}{p^{1+\gamma+\delta}}\right)}
\end{align}
and $A_{p}(\alpha,\beta,\gamma,\delta):= G_{p}(\alpha,\beta,\gamma,\delta)/Y_{p}(\alpha,\beta,\gamma,\delta)$
 is another local function, which converges like $1/p^{2}$ for sufficient small $\text{Re}(\alpha),\text{Re}(\beta),\text{Re}(\gamma),$ and $\text{Re}(\delta)$.\\
\\
Next, suppose $p \equiv 1(4)$.  Factoring out terms with slow convergence as above, we expand $G_{p}(\alpha,\beta,\gamma,\delta)$ as

\begin{align}
\begin{split}
&G_{p}(\alpha,\beta,\gamma,\delta)=\sum_{\substack{m+n \\ \text{even}}}\bigg ( \frac{\min{\{m,n\}}+1}{p^{n(\frac 1 2+\alpha)+m(\frac 1 2+\beta)}}+\frac{\min{\{m,n\}}+1}{p^{(1+2\gamma)+(1+2\delta)+n(\frac 1 2+\alpha)+m(\frac 1 2+\beta)}}\\
&\phantom{=}+ \frac{\min{\{m,n\}}+1}{p^{(1+2\delta)+n(\frac 1 2+\alpha)+m(\frac 1 2+\beta)}}+\frac{\min{\{m,n\}}+1}{p^{(1+2\gamma)+n(\frac 1 2+\alpha)+m(\frac 1 2+\beta)}}\bigg )\\
&\phantom{=}+
\sum_{\substack{m+n \\ \text{odd}}}\bigg ( \frac{-2(\min{\{m,n\}}+1)}{p^{(\frac 1 2+\delta)+n(\frac 1 2+\alpha)+m(\frac 1 2+\beta)}}+\frac{-2(\min{\{m,n\}}+1)}{p^{(\frac 1 2+\gamma)+n(\frac 1 2+\alpha)+m(\frac 1 2+\beta)}}\\
&\phantom{=}+ \frac{-2(\min{\{m,n\}}+1)}{p^{(\frac 1 2+\gamma)+2(\frac 1 2+\delta)+n(\frac 1 2+\alpha)+m(\frac 1 2+\beta)}}+\frac{-2(\min{\{m,n\}}+1)}{p^{2(\frac 1 2+\gamma)+(\frac 1 2+\delta)+n(\frac 1 2+\alpha)+m(\frac 1 2+\beta)}}\bigg )\\
&\phantom{=}+
\sum_{n}\bigg ( \frac{4n+2}{p^{(1+\gamma+\delta)+n(1+\alpha+\beta)}}\bigg )+\sum_{\substack{m+n \\ \text{even}\\ m\neq n}}\frac{4\left(\min{\left\{m,n\right\}}+1\right)}{p^{(1+\gamma+\delta)+n(\frac 1 2+\alpha)+m(\frac 1 2+\beta)}}\\
&=\left ( \sum_{\substack{m+n \\ \text{ even}}} \frac{\min\{m,n\} +1}{p^{n(\frac 1 2+\alpha)+m(\frac 1 2+\beta)}} \right )\left( 1+ \frac{1}{p^{1+2\gamma}}+ \frac{1}{p^{1+2\delta}} + \frac{1}{p^{2+2\gamma + 2\delta}} \right )\\
&\phantom{=}+ \left ( \sum_{\substack{m+n \\ \text{ odd}}} \frac{-2(\min\{m,n\}+1)}{p^{n(\frac 1 2+\alpha)+m(\frac 1 2+\beta)}} \right )\times\left ( \frac{1}{p^{\frac 1 2+\gamma}} +\frac{1}{p^{\frac 1 2+\delta}} + \frac{1}{p^{\frac 3 2+2\gamma +\delta}} + \frac{1}{p^{\frac 3 2+\gamma+2\delta}} \right )\\
&\phantom{=}+\left ( \sum_{\substack{m+n \\ \text{ even} \\ m \neq n}} \frac{4 \min\{m,n\}+4}{p^{n(\frac 1 2+\alpha)+m(\frac 1 2+\beta)}} + \sum_{n} \frac{4 n+2}{p^{n(1+\alpha+\beta)}} \right ) \left ( \frac{1}{p^{1+\gamma+\delta}} \right ).
\end{split}
\end{align}

Since
\begin{equation}
 \sum_{\substack{m+n\\ \text{ even}}} \frac{\min\{m,n\} +1}{p^{n(\frac 1 2+\alpha)+m(\frac 1 2+\beta)}} = \left(1+\frac{1}{p^{1+2\alpha}}+\frac{1}{p^{1+2\beta}}+\frac{2}{p^{1+\alpha+\beta}}+O\left(\frac{1}{p^2}\right)\right), \\
\end{equation}
\begin{equation}
    \sum_{\substack{m+n\\ \text{ odd}}} \frac{-2(\min\{m,n\}+1)}{p^{n(\frac 1 2+\alpha)+m(\frac 1 2+\beta)}} = \left (\frac{-2}{p^{\frac 1 2+\alpha}} +\frac{-2}{p^{\frac 1 2+\beta}}+O\left ( \frac{1}{p^{\frac 3 2}} \right )
    \right ), \\
\end{equation}
and
\begin{align}
\left ( \sum_{\substack{m+n \\ \text{ even} \\ m \neq n}} \frac{4 \min\{m,n\}+4}{p^{n(\frac 1 2+\alpha)+m(\frac 1 2+\beta)}} + \sum_{n} \frac{4n+2}{p^{n(1+\alpha+\beta)}} \right ) \left ( \frac{1}{p^{1+\gamma+\delta}} \right ) &= \frac{2}{p^{1+\gamma+\delta}} + O \left ( \frac{1}{p^2} \right ),
\end{align}
we conclude that, for $p \equiv 1(4)$, we may write
\begin{equation}
G_{p}(\alpha,\beta,\gamma,\delta) = Y_{p}(\alpha,\beta,\gamma,\delta)\times A_{p}(\alpha,\beta,\gamma,\delta),
\end{equation}
where
\begin{align}
Y_{p}(\alpha,\beta,\gamma,\delta):=&\frac{\left(1-\frac{1}{p^{1+\alpha+\gamma}}\right)\left(1-\frac{1}{p^{1+\beta+\gamma}}\right)
\left(1-\frac{1}{p^{1+\alpha+\delta}}\right)\left(1-\frac{1}{p^{1+\beta+\delta}}\right)}{\left(1-\frac{1}{p^{1+2\alpha}}\right)\left(1-\frac{1}{p^{1+2\beta}}\right)\left(1-\frac{1}{p^{1+\alpha+\beta}}\right)\left(1-\frac{1}{p^{1+\gamma+\delta}}\right)}\nonumber 
\\
& \times\frac{\left(1-\frac{1}{p^{1+\alpha+\gamma}}\right)\left(1-\frac{1}{p^{1+\beta+\gamma}}\right)
\left(1-\frac{1}{p^{1+\alpha+\delta}}\right)\left(1-\frac{1}{p^{1+\beta+\delta}}\right)}{\left(1-\frac{1}{p^{1+\alpha+\beta}}\right)\left(1-\frac{1}{p^{1+2\gamma}}\right)\left(1-\frac{1}{p^{1+2\delta}}\right)\left(1-\frac{1}{p^{1+\gamma+\delta}}\right)},
\end{align}
and $A_{p}(\alpha,\beta,\gamma,\delta)$ is a function that converges sufficiently rapidly.\\
\\
Finally, note that
\begin{align}
\begin{split}
G_{2}(\alpha,\beta,\gamma,\delta)&=\sum_{\substack{m+n \\ \text{even}}}\left(\frac{\delta_{2}(m,n,0,0)}{2^{n(\frac 1 2+\alpha)+m(\frac 1 2+\beta)}}+\frac{\delta_{2}(m,n,1,1)}{2^{(\frac 1 2+\gamma)+(\frac 1 2+\delta)+n(\frac 1 2+\alpha)+m(\frac 1 2+\beta)}}\right)\\
&\phantom{=}+\sum_{\substack{m+n \\ \text{odd}}}\left(\frac{\delta_{2}(m,n,1,0)}{2^{(\frac 1 2+\gamma)+n(\frac 1 2+\alpha)+m(\frac 1 2+\beta)}}+\frac{\delta_{2}(m,n,0,1)}{2^{(\frac 1 2+\delta)+n(\frac 1 2+\alpha)+m(\frac 1 2+\beta)}}\right)\\
&=\left(1+\frac 1 {2^{1+\gamma+\delta}}\right)\sum_{\substack{m+n \\ \text{even}}}\left(\frac{1}{2^{n(\frac 1 2+\alpha)+m(\frac 1 2+\beta)}}\right)\\
&\phantom{=}-\left(\frac 1 {2^{\frac 1 2+\gamma}}+\frac 1 {2^{\frac 1 2+\delta}}\right)\sum_{\substack{m+n \\ \text{odd}}}\left(\frac{1}{2^{n(\frac 1 2+\alpha)+m(\frac 1 2+\beta)}}\right).
\end{split}
\end{align}

We therefore may write
\begin{equation}
G_{2}(\alpha,\beta,\gamma,\delta) = Y_{2}(\alpha,\beta,\gamma,\delta)\times A_{2}(\alpha,\beta,\gamma,\delta),
\end{equation}
where
\begin{align}
\begin{split}
Y_{2}(\alpha,\beta,\gamma,\delta)&:=\frac{\left(1-\frac{1}{2^{1+\alpha+\gamma}}\right)\left(1-\frac{1}{2^{1+\beta+\gamma}}\right)
\left(1-\frac{1}{2^{1+\alpha+\delta}}\right)\left(1-\frac{1}{2^{1+\beta+\delta}}\right)}{\left(1-\frac{1}{2^{1+2\alpha}}\right)\left(1-\frac{1}{2^{1+2\beta}}\right)\left(1-\frac{1}{2^{1+\alpha+\beta}}\right)\left(1-\frac{1}{2^{1+\gamma+\delta}}\right)}
\end{split}
\end{align}
and $A_{2}(\alpha,\beta,\gamma,\delta):= G_{2}(\alpha,\beta,\gamma,\delta)/Y_{2}(\alpha,\beta,\gamma,\delta)$.
\subsection{Expanding the Euler Product}

Recall that for Re$(x)> 0$,
\begin{align}
\zeta(1+x) &= \prod_p \left(1-\frac{1}{p^{1+x}}\right)^{-1}, \label{zeta+}
\end{align}
and
\begin{align}
\begin{split}
L(1+x) &= \prod_{p \equiv 1(4)} \left(1- \frac{1}{p^{1+x}} \right)^{-1}\prod_{p \equiv 3(4)} \left(1+ \frac{1}{p^{1+x}}\right)^{-1},
\end{split}
\end{align}
where $L(s):=L(s,\chi_{1})$.  Incorporating the above simplifications, and again collecting only terms which converge substantially slower that $p^{-3/2}$, we arrive at the following conjecture.
\begin{conjecture} With constraints on $\alpha, \beta, \gamma, \delta$ as described in (\ref{ratios domain}) and (\ref{ratios domain 2}), we have
\begin{align}
\begin{split}
&\sum_{k\neq 0}\bigg | \widehat{f}\left( \frac{k}{K}\right) \bigg |^{2}\frac{L_{k}(\frac 1 2+\alpha)L_{k}(\frac 1 2+\beta)}{L_{k}(\frac 1 2+\gamma)L_{k}(\frac 1 2+\delta)}=\sum_{k \neq 0}\bigg | \widehat{f} \left( \frac{k}{K}\right) \bigg |^{2}\bigg (G(\alpha,\beta,\gamma,\delta)\\
&\phantom{=}+\frac{1}{1-2\alpha}\left(\frac{\pi}{2K}\right)^{2\alpha}G(-\alpha,\beta,\gamma,\delta)+\frac{1}{1-2\beta}\left(\frac{\pi}{2K}\right)^{2\beta}G(\alpha,-\beta,\gamma,\delta)\\
&\phantom{=}+\left(\frac{1}{1-2(\alpha+\beta)}\right)\left(\frac{\pi}{2K}\right)^{2(\alpha+\beta)}G(-\alpha,-\beta,\gamma,\delta)\bigg )+O\left(K^{\frac{1}{2}+\epsilon}\right),
\end{split}
\end{align}
where
 
\begin{align}\label{G factor}
G(\alpha,\beta,\gamma,\delta)&:=\prod_{p}G_{p}(\alpha,\beta,\gamma,\delta) \\
&\hspace{1mm}= Y(\alpha,\beta,\gamma,\delta)\times A(\alpha,\beta,\gamma,\delta),
\end{align}

\begin{align}
\begin{split}
Y(\alpha,\beta,\gamma,\delta) &:= \prod_{p}Y_{p}(\alpha,\beta,\gamma,\delta)\\
&= \frac{\zeta(1+2\alpha)\zeta(1+2\beta)\zeta(1+\gamma+\delta)\zeta(1+\alpha+\beta)}{\zeta(1+\alpha+\gamma)\zeta(1+\beta+\gamma)\zeta(1+\beta+\delta)\zeta(1+\alpha+\delta)}\\
&\hspace{5mm}\times \frac{L(1+2\gamma)L(1+2\delta)L(1+\gamma+\delta)L(1+\alpha+\beta)}{L(1+\alpha+\gamma)L(1+\beta+\gamma)L(1+\beta+\delta)L(1+\alpha+\delta)},
\end{split}
\end{align}
and $ A(\alpha,\beta,\gamma,\delta):= \prod_{p}A_{p}(\alpha,\beta,\gamma,\delta) $ is an Euler product that converges for sufficiently small fixed values of $\textnormal{Re}(\alpha),\textnormal{Re}(\beta),\textnormal{Re}(\gamma),\textnormal{Re}(\delta)$.
\end{conjecture}

In further calculations, it will be helpful to define
\begin{equation}
\mathcal{Y}(\alpha,\beta,\gamma,\delta):=\frac{\zeta(1+2\alpha)\zeta(1+2\beta)\zeta(1+\gamma+\delta)\zeta(1+\alpha+\beta)}{\zeta(1+\alpha+\gamma)\zeta(1+\beta+\gamma)\zeta(1+\beta+\delta)\zeta(1+\alpha+\delta)},
\end{equation}
as well as 
\begin{equation}
\mathcal{A}(\alpha,\beta,\gamma,\delta):=\frac{G(\alpha,\beta,\gamma,\delta)}{\mathcal{Y}(\alpha,\beta,\gamma,\delta)}.
\end{equation}
It will also be necessary to make use of the following lemma.
\begin{lemma}\label{A is 1}
We have that 
\begin{align}
A(\alpha,\beta,\alpha,\beta)=\mathcal{A}(\alpha,\beta,\alpha,\beta)=1.
\end{align}
\end{lemma}
\begin{proof}
Since $Y(\alpha,\beta,\alpha,\beta)=\mathcal{Y}(\alpha,\beta,\alpha,\beta)=1$, it suffices to show that
$\hfill \break G(\alpha,\beta,\alpha,\beta)=1$.  Note that $G_{2}(\alpha,\beta,\alpha,\beta )=1$, and upon writing

\begin{equation}
\sum_{m, n} \frac{1}{p^{n(1+2\alpha)+m(1+2\beta)}}=\left(1-\frac{1}{p^{1+2\beta}}\right)^{-1}
\left(1-\frac{1}{p^{1+2\alpha}}\right)^{-1},
\end{equation}
we similarly obtain that $G_{p}(\alpha,\beta,\alpha,\beta )=1$ whenever $p \equiv 3(4)$.  Moreover, we rewrite
\begin{align}
\sum_{\substack{m+n\\ \text{even}}} \frac{\text{min}(m,n) +1}{p^{m(\frac{1}{2}+\alpha)+n(\frac{1}{2}+\beta)}}= \frac{p^{2(1+\alpha +\beta)}(1+p^{1+\alpha+\beta})}{(p^{1+2\alpha}-1)(p^{1+\alpha+\beta}-1)(p^{1+2\beta}-1)},
\end{align}
and
\begin{equation}
\sum_{\substack{m+n \\ \text{odd}}} \frac{-2(\text{min}(m,n) +1)}{p^{m(\frac{1}{2}+\alpha)+n(\frac{1}{2}+\beta)}}=\frac{-2p^{\frac{5}{2}+2\alpha+2\beta}(p^{\alpha}+p^{\beta})}{(p^{1+2\alpha}-1)(p^{1+\alpha+\beta}-1)(p^{1+2\beta}-1)},
\end{equation}
as well as
\begin{align}
\begin{split}
\sum_{\substack{m \neq n \\ m+n \text{ even}}} \frac{4\cdot \text{min}(m, n)+4}{p^{m(\frac{1}{2}+\alpha)+n(\frac{1}{2}+\beta)}} &= \frac{4p^{2(1+\alpha+\beta)}(1+p^{1+\alpha+\beta})}{(p^{1+2\alpha}-1)(p^{1+\alpha+\beta}-1)(p^{1+2\beta}-1)}\\
&\phantom{=}-\frac{4p^{2(1+\alpha+\beta)}}{(p^{1+\alpha+\beta}-1)^{2}},
\end{split}
\end{align}
and

\begin{equation}
\sum_{n=0}^\infty \frac{4n+2}{p^{n(1+\alpha+\beta)}}=\frac{2p^{1+\alpha+\beta}(1+p^{1+\alpha+\beta})}{(p^{1+\alpha+\beta}-1)^2},
\end{equation}
so that for $p \equiv 1(4)$,
\begin{align}
\begin{split}
&G_{p}(\alpha,\beta,\gamma,\delta)=\left (  \frac{p^{2(1+\alpha +\beta)}(1+p^{1+\alpha+\beta})}{(p^{1+2\alpha}-1)(p^{1+\alpha+\beta}-1)(p^{1+2\beta}-1)} \right )\\
& \times\bigg(1+ \frac{1}{p^{1+2\gamma}}+ \frac{1}{p^{1+2\delta}} + \frac{1}{p^{2+2\gamma + 2\delta}} \bigg)- \left (\frac{2p^{\frac{5}{2}+2\alpha+2\beta}(p^{\alpha}+p^{\beta})}{(p^{1+2\alpha}-1)(p^{1+\alpha+\beta}-1)(p^{1+2\beta}-1)} \right )\\
&\times \left ( \frac{1}{p^{\frac 1 2+\gamma}} +\frac{1}{p^{\frac 1 2+\delta}} + \frac{1}{p^{\frac 3 2+2\gamma +\delta}} + \frac{1}{p^{\frac 3 2+\gamma+2\delta}} \right )+\bigg(\frac{2p^{1+\alpha+\beta}(1+p^{1+\alpha+\beta})}{(p^{1+\alpha+\beta}-1)^2}\\
&+ \frac{4p^{2+2\alpha+2\beta}(1+p^{1+\alpha+\beta})}{(p^{1+2\alpha}-1)(p^{1+\alpha+\beta}-1)(p^{1+2\beta}-1)} -\frac{4p^{2+2\alpha+2\beta}}{(p^{1+\alpha+\beta}-1)^{2}}\bigg) \left ( \frac{1}{p^{1+\gamma+\delta}} \right )\nonumber.
\end{split}
\end{align}
Upon setting $\alpha = \gamma$ and $\beta = \delta$, we then have $G_{p}(\alpha,\beta,\alpha,\beta)=1$.  The lemma then follows from (\ref{G factor}).
\end{proof}
\begin{lemma}\label{A integral}
Define $A_{\beta}(\alpha):=A(-\alpha,-\beta,\alpha,\beta).$ Then 
\begin{align}
\frac{d}{d \alpha}A_{\beta}(\alpha)\bigg \vert_{\alpha = -\beta} &= -2\sum_{p\equiv 3(4)}\frac{\left(p^{2+8\beta}+p^2-2 p^{4\beta}\right) \log p}{p^{2+8\beta}+p^2-p^{4\beta}-p^{4+4\beta}}.
\end{align}
\end{lemma}
\begin{proof}
Write
\begin{align}
A_{\beta}(\alpha) = \prod_{p} p_{\beta}(\alpha),
\end{align}
where

\begin{equation}
p_{\beta}(\alpha):= A_{p}(-\alpha,-\beta,\alpha,\beta)
\end{equation}

 are the local factors of $A_{\beta}(\alpha)$, and note that $p_{\beta}(-\beta)=1$ at each prime $p$.  By the product rule,
\begin{equation}
\frac{d}{d \alpha}A_{\beta} = \sum_{q}\frac{d}{d \alpha}p_{\beta}\prod_{p \neq q}q_{\beta}.
\end{equation}

Note that
\begin{equation}
    p_{\beta}(\alpha)=\left\{
\begin{array}{l l l}
\frac{2^{-\alpha-\beta}(-2+2^{\alpha+\beta})(-1+2^{1+\alpha+\beta})(2-2^{1+2\alpha}+2^{\alpha+\beta}-2^{1+2\beta}+2^{1+2\alpha+2\beta})(-2+2^{2\alpha})(-2+2^{2\beta})}{(2^{1/2}-2^{\alpha})(2^{1/2}+2^{\alpha})(2^{1/2}-2^{\beta})(2^{1+\alpha}-2^\beta)(2^{1/2}+2^\beta)(2^\alpha-2^{1+\beta})} & \text{ if } p=2 \\
-\frac{1}{(-1+p)^4(p^{1+\alpha}-p^\beta)^2(p^\alpha-p^{1+\beta})^2}p^{-4(\alpha+\beta)}(-1+p^{1+2\alpha})(-p+p^{\alpha+\beta}) & \text{ if } p\equiv 1(4)\\
\quad \times (-1+p^{1+\alpha+\beta})^2(-1+p^{1+2\beta})(p-2p^{1+2\alpha}+p^{2+2\alpha}+p^{\alpha+\beta}-2p^{3(\alpha+\beta)} & \\
\quad -4p^{1+\alpha+\beta}-4p^{2(1+\alpha+\beta)}+2p^{2+\alpha+\beta}+3p^{1+3\alpha+\beta}-2p^{2+3\alpha+\beta}-2p^{1+2\beta} & \\
\quad +p^{2+2\beta}+4p^{1+2\alpha+2\beta}+p^{3+2\alpha+2\beta}+3p^{1+\alpha+3\beta}-2p^{2+\alpha+3\beta}+p^{2+3\alpha+3\beta}) & \\
\frac{p^{-4(\alpha+\beta)}(-1+p^{1+2\alpha})(1+p^{1+2\alpha})(-p+p^{\alpha+\beta})(p+p^{\alpha+\beta})(-1+p^{1+\alpha+\beta})}{(-1+p^2)^2(p^{2+4\alpha}-p^{2(\alpha+\beta)}-p^{2(2+\alpha+\beta)}+p^{2+4\beta}}) & \text{ if } p\equiv 3(4), \\
\quad \times (1+p^{1+\alpha+\beta})(-1+p^{1+2\beta})(1+p^{1+2\beta}) & \end{array} \right.
\end{equation}
so that
\begin{equation}
\frac{d}{d \alpha}p_{\beta}(\alpha)\bigg \vert_{\alpha = -\beta} = \left\{
\begin{array}{l l l}
0 & \text{ if } p=2 \\
0 & \text{ if } p\equiv 1(4)\\
-2\frac{\left(p^{2+8 \beta}+p^2-2 p^{4 \beta}\right) \log p}{p^{2+8 \beta}+p^2-p^{4 \beta }-p^{4+4 \beta }} & \text{ if } p\equiv 3(4),
\end{array} \right.
\end{equation}
from which the result follows.
\end{proof}
\section{The Ratios Conjecture Prediction for $\text{Var}(\psi_{K,X})$:}
Let $F_{K}(\theta)$ be as in (\ref{Fk}). By the Fourier expansion of $F_{K}$, we may write
\begin{align}
\begin{split}
\psi_{K,X}(\theta) &=\sum_{\A } \Phi \left(\frac{N(\A)}{X}\right) \Lambda(\A) F_K(\theta_{\A}  -\theta)\\
&= \sum_{\A }\Phi\left(\frac{N(\A)}{X}\right)\Lambda(\A)\sum_{k \in \Z}\frac{1}{K}\widehat{f}\left(\frac{k}{K}\right) e^{4ki(\theta_{\A}-\theta)}.
\end{split}
\end{align}
Since the mean value is given by the zero mode $k=0$, the variance may be computed as
\begin{align}
\begin{split}
\text{Var}(\psi_{K,X})&=\frac{2}{\pi}\int_{0}^{\frac \pi 2}\bigg|\psi_{K,X}(\theta) - \langle \psi_{K,X}\rangle\bigg|^{2}d\theta \\
&=\frac{2}{\pi}\int_{0}^{\frac \pi 2}\bigg|\sum_{k\neq 0} e^{-i4k\theta} \frac 1{ K}\widehat{f}\left(\frac {k}{ K}\right) \sum_{\A} \Phi\left(\frac{N(\A)}{X}\right) \Lambda(\A)\Xi_{k}(\A)\bigg|^{2}d\theta.\label{nonk}
\end{split}
\end{align}
By applying the \textit{Mellin Inversion Formula}
\begin{equation}
\Phi(x) = \frac{1}{2\pi i}\int_{(2)}\tilde{\Phi}(s)x^{-s}ds,
\end{equation}
we obtain
\begin{align}
\begin{split}
 \sum_{\mathfrak a} \Lambda(\mathfrak a) \Xi_{k}(\mathfrak a)\Phi\left(\frac{N(\mathfrak a)}{X}\right) &= \frac 1{2\pi i}\int_{(2)}  \sum_{\mathfrak a} \Lambda(\mathfrak a) \Xi_{k}(\A) \frac{X^s}{N(\A)^s} \tilde \Phi(s) ds\\
&=  \frac 1{2\pi i}\int_{(2)} -\frac{L_{k}'}{L_{k}}(s) \tilde \Phi(s) X^s ds.
\end{split}
\end{align}
Inserting this into (\ref{nonk}), we find that
\begin{align}
\begin{split}
\text{Var}(\psi_{K,X}) &= \frac{2}{\pi}\int_{0}^{\frac \pi 2}\bigg|\sum_{k\neq 0} e^{-i4k\theta} \frac 1{ K}\widehat{f}\left(\frac {k}{ K}\right) \sum_{\A} \Lambda(\A)\Xi_{k}(\A)\Phi\left(\frac{N(\A)}{X}\right)\bigg|^{2}d\theta\\
&= \frac{2}{\pi}\int_{0}^{\frac \pi 2}\bigg|\sum_{k\neq 0} e^{-i4k\theta} \frac 1{ K}\widehat{f}\left(\frac {k}{ K}\right) \frac {i}{2\pi}\int_{(2)}\frac{L_{k}'}{L_{k}}(s) \tilde \Phi(s) X^s ds\bigg|^{2}d\theta.
\end{split}
\end{align}
Upon recalling that 
\begin{align}
\int_{0}^{\frac \pi 2}e^{4i(k'-k)\theta}d\theta =  \left\{
\begin{array}{l l}
0 & \text{ if } k \neq k' \\
\frac \pi 2 & \text{ if } k = k',\\
\end{array} \right.
\end{align}
$\text{Var}(\psi_{K,X})$ can be restricted to terms for which the Fourier coefficients are equal, i.e.,
\begin{align}
\begin{split}
\text{Var}(\psi_{K,X}) &=\frac 1{4\pi^2K^{2}}\int_{(2)}\int_{(2)}\sum_{k\neq 0}\bigg |\widehat{f}\left(\frac {k}{ K}\right)\bigg |^{2}\frac{L_{k}'}{L_{k}}(s)\frac{L_{k}'}{L_{k}}(\overline{s'}) \tilde \Phi(s)\tilde \Phi(\overline{s'})X^{s} X^{\overline{s'}}ds\overline{ds'}
\end{split}
\end{align} 
by Fubini's theorem.  Moreover, under GRH, $\frac {L_{k}'}{L_{k}}(s)$ is holomorphic in the half-plane Re$(s)>\frac {1}{2}$, and thus we may shift the vertical integrals to Re$(s)=\frac{1}{2}+\epsilon$, and Re$(s')=\frac{1}{2}+\epsilon'$, for any $\epsilon, \epsilon'>0$.  Upon making the change of variables $\alpha:= s-\frac{1}{2}$ and $\beta:= s'-\frac 1 2$ we find that
\begin{equation}\label{varcontour}
\begin{split}
\text{Var}(\psi_{K,X})&=- \frac {X^{1-2\lambda}}{4\pi^2}\int_{(\epsilon')}\int_{(\epsilon)}\sum_{k\neq 0} \bigg |\widehat{f}\left(\frac {k}{ K}\right)\bigg |^{2}\frac{L_{k}'}{L_{k}}\left(\frac{1}{2}+\alpha\right)\frac{L_{k}'}{L_{k}}\left(\frac{1}{2}+\beta\right)\\
&\phantom{=}\times \tilde \Phi\left(\frac{1}{2}+\alpha\right)\tilde \Phi\left(\frac{1}{2}+\beta\right) X^{\beta}X^{\alpha}d\alpha d\beta.
\end{split}
\end{equation}
Note by (\ref{ratios domain 2}) that the substitution of the ratios conjecture is only valid when Im$(\alpha)$, Im$(\beta)\ll_{c} K^{1-c}$, for small $c>0$.  If either Im$(\alpha)> K^{1-c}$ or Im$(\beta)> K^{1-c}$, we use the rapid decay of $\tilde{\Phi}$, as well as upper bounds on the growth of $\frac{L'_{k}}{L_{k}}$ within the critical strip, to show that the contribution to the double integral coming from these tails is bounded by $O_{c}\left(K^{-1+c}\right)$.    For Im$(\alpha)$, Im$(\beta) < K^{1-c}$, we take the derivative of (\ref{the conjecture M}) to obtain
\begin{align}\label{M prime}
\sum_{\substack{k \neq 0}} \bigg |\widehat{f}\left(\frac {k}{ K}\right)\bigg |^{2}\frac{L_{k}'}{L_{k}}\left(\frac{1}{2}+\alpha\right)\frac{L_{k}'}{L_{k}}\left(\frac{1}{2}+\beta\right)= \sum_{\substack{k \neq 0}} \bigg |\widehat{f}\left(\frac {k}{ K}\right)\bigg |^{2}M'_{K}(\alpha,\beta)+O\left(K^{\frac{1}{2}+\epsilon}\right),
\end{align}

where\footnote{Here, and elsewhere, we allow for a slight abuse of notation: $\alpha$ and $\beta$ denote coordinates of $M_{K}$, as well as coordinates of the point at which the derivative is then evaluated.}
\begin{align}
\begin{split}
&M_{K}'(\alpha,\beta) := \frac{\partial }{\partial \beta}\frac{\partial}{\partial \alpha}M_{K}(\alpha,\beta,\gamma,\delta)\Bigg \vert_{(\alpha,\beta,\alpha,\beta)}.
\end{split}
\end{align}

Plugging (\ref{M prime}) into (\ref{varcontour}) for Im$(\alpha)$, Im$(\beta) < K^{1-c}$, and using a similar argument as above to bound the tails, we then arrive at the following conjecture:
\begin{conjecture}\label{full conjecture}  We have that
\begin{align}
&\text{Var}(\psi_{K,X}) = -C_{f}\frac {X}{K}\bigg(I_{1} + I_{2}+I_{3}+I_{4}\bigg)+O\left(X^{-\frac{\lambda}{2}+\epsilon}\right),
\end{align}
where
\begin{equation}
\begin{split}
I_{1}:&= \int_{(\epsilon')} \int_{(\epsilon)}\frac{\partial }{\partial \beta}\frac{\partial}{\partial \alpha} G(\alpha,\beta,\gamma,\delta)\bigg \vert_{(\alpha,\beta,\alpha,\beta)}\\
&\phantom{=}\times \tilde \Phi\left(\frac{1}{2}+\alpha\right)\tilde \Phi\left(\frac{1}{2}+\beta\right)X^{\alpha+\beta}d\alpha d\beta,
\end{split}
\end{equation}

\begin{align}
\begin{split}
I_{2}:&=\int_{(\epsilon')}\int_{(\epsilon)}\frac{\partial }{\partial \beta}\frac{\partial}{\partial \alpha}\left(\frac{\pi}{2}\right)^{2\beta}\frac{1}{1-2\beta} G(\alpha,-\beta,\gamma,\delta)\bigg \vert_{(\alpha,\beta,\alpha,\beta)}\\
&\phantom{=}\times \tilde \Phi\left(\frac{1}{2}+\alpha\right)\tilde \Phi\left(\frac{1}{2}+\beta\right) X^{\alpha}X^{\beta(1-2\lambda)}d\alpha d\beta 
\end{split}
\end{align}

\begin{align}
\begin{split}
I_{3}:&= \int_{(\epsilon')}\int_{(\epsilon)}\frac{\partial }{\partial \beta}\frac{\partial}{\partial \alpha}\left(\frac{\pi}{2}\right)^{2\alpha}\frac{1}{1-2\alpha} G(-\alpha,\beta,\gamma,\delta)\bigg \vert_{(\alpha,\beta,\alpha,\beta)}\\
&\phantom{=} \times \tilde \Phi\left(\frac{1}{2}+\alpha\right)\tilde \Phi\left(\frac{1}{2}+\beta\right) X^{\alpha(1-2\lambda)}X^{\beta}d\alpha d\beta
\end{split}
\end{align}
and
\begin{align}
\begin{split}
I_{4}&:=\int_{(\epsilon')}\int_{(\epsilon)}\frac{\partial }{\partial \beta}\frac{\partial}{\partial \alpha} \left(\frac{1}{1-2(\alpha+\beta)}\right)\left(\frac{\pi}{2}\right)^{2(\alpha+\beta)} G(-\alpha,-\beta,\gamma,\delta)\bigg \vert_{(\alpha,\beta,\alpha,\beta)}\\
&\hspace{5mm}\times\tilde \Phi\left(\frac{1}{2}+\alpha\right) \tilde \Phi\left(\frac{1}{2}+\beta \right) X^{\alpha(1-2\lambda)}X^{\beta(1-2\lambda)}d\alpha d\beta.
\end{split}
\end{align}
\end{conjecture}

Conjecture \ref{conj} now follows from Conjecture \ref{full conjecture} as a consequence of the following three lemmas:
\begin{lemma}\label{1st total} We have
\begin{equation}
I_{1} = -(\log X) C_{\Phi}-C'_{\Phi}-\pi^2\tilde{\Phi}\left(\frac{1}{2}\right)^2+O_{\Phi}\left(X^{-\frac 1 5}\right).
 \end{equation}
\end{lemma}

\begin{lemma} \label{second total} 
We have
\begin{equation}
I_{2}+I_{3} = \left\{
\begin{array}{l l l}
O_{\Phi}\left(X^{-\epsilon}\right)  & \text{ if } \lambda > 1 \\
2 \pi^2\left(\tilde{\Phi}\left(\frac 1 2\right)\right)^{2} +O_{\Phi}\left(X^{-\epsilon}\right) & \text{ if } \frac 1 2 < \lambda < 1 \\
4 \pi^2\left(\tilde{\Phi}\left(\frac 1 2\right)\right)^{2} + O_{\Phi}\left(X^{-\epsilon}\right)& \text{ if } \lambda < \frac 1 2,
\end{array}\right.
\end{equation}
where $\epsilon > 0$ is a constant $($depending on $\lambda)$.
\end{lemma}
\begin{lemma}\label{fourth total}
We have
\begin{equation}
I_{4} = \left\{\begin{array}{l l}
C_{\Phi}(1 - 2\lambda)\log X+\kappa+ O_{\Phi}\left(X^{-\epsilon}\right)& \text{ if } \frac 1 2  < \lambda\\
O_{\Phi}\left(X^{-\epsilon}\right) & \text{ if } \frac 1 2  > \lambda,
\end{array}\right.
\end{equation}

where
\begin{equation}
\kappa := C_{\Phi}\left(\log \left(\frac {\pi^2}{4}\right)+2\right)+C_{\Phi,\zeta}-C_{\Phi,L} +C'_\Phi-\pi^2\left(\tilde{\Phi}\left(\frac 1 2\right)\right)^{2}-A_{\Phi}'.
\end{equation}

Here $\epsilon > 0$ is a constant $($depending on $\lambda)$, and $C_{\Phi,\zeta}$, $C_{\Phi,L}$, and $A_{\Phi}'$, are as in $(\ref{zeta constant})$, $(\ref{L constant})$, and $(\ref{A constant})$, respectively.
\end{lemma}
Conjecture \ref{conj} follows upon inserting the results from Lemma \ref{1st total}, Lemma \ref{second total}, and Lemma \ref{fourth total}, into Conjecture \ref{full conjecture}.  Note that when $\lambda >1$, Conjecture \ref{full conjecture} moreover agrees with Theorem \ref{trivial regime theorem}.

\section{Auxiliary Lemmas}
Before proceeding to the proofs of Lemmas \ref{1st total}, \ref{second total}, and \ref{fourth total}, we will prove a few auxiliary lemmas that will be used frequently in the rest of the paper.
\begin{lemma}\label{countour bound}
Let $h(\alpha)$ be holomorphic in $\Omega := \left\{-\frac{1}{4}< \textnormal{Re}(\alpha)< \epsilon\right\}$ for some $\epsilon>0$, except for possibly at a finite set of poles.  Moreover, suppose that $h(\alpha)$ does not grow too rapidly in $\Omega$, i.e., there exists a fixed $d>0$ such that $h(\alpha) \ll |\alpha|^{d}$ away from the poles in $\Omega$.  Set
\begin{equation}
f(\alpha) := h(\alpha)\tilde{\Phi}\left(\frac 1 2+\alpha \right)X^{\alpha},
\end{equation}
where $\alpha, \beta,$ and $\tilde{\Phi}$ are as above.  Then

\begin{equation}
\int_{(\epsilon)} f(\alpha)d\alpha =2\pi i\cdot \sum_{k}\textnormal{Res}(f,a_{k})+O\left( X^{-\frac{1}{5}}\right),
\end{equation}

where $\textnormal{Res}(f,a_{k})$ denotes the residue of $f$ at each pole $a_{k} \in \Omega$.
\end{lemma}

\begin{proof}
Consider the contour integral drawn counter-clockwise along the closed box
\begin{equation}
C_{T} := V_{1} \cup H_{1} \cup V_{2} \cup  H_{2},
\end{equation}
where
\begin{align}
\left\{
\begin{array}{l l}
V_{1} &:= [\epsilon-iT,\epsilon+iT]\\
H_{1} &:= [\epsilon+iT,-\frac 1 4+\epsilon+iT]\\
V_{2} &:= [-\frac 1 4+\epsilon+iT,-\frac 1 4+\epsilon-iT]\\
H_{2} &:= [-\frac 1 4+\epsilon-iT,\epsilon-iT].
\end{array} \right.
\end{align}

By Cauchy's residue theorem,

\begin{align}
\begin{split}
\int_{(\epsilon)}f(\alpha) ~d\alpha &= 2\pi i \cdot \sum_{k}\textnormal{Res}(f,a_{k})- \lim_{T \rightarrow \infty}\bigg ( \int_{H_{1} \cup V_{2} \cup H_{2}} f(\alpha)d\alpha\bigg ).
\end{split}
\end{align}

Set $\alpha = \sigma+iT$.  By the properties of the Mellin transform, we find that for any fixed $A >0$,
\begin{equation}
\tilde{\Phi}\left(\frac 1 2 +it \right) \ll \textnormal{min}(1,|t|^{-A}).
\end{equation}
Since moreover $h(\alpha)$ does not grow too rapidly, we bound
\begin{align}
\begin{split}
\int_{H_{1}} f(\alpha)d\alpha &= \int_{\epsilon}^{-1/4+\epsilon}h(\sigma+iT)\tilde{\Phi}\left(\frac 1 2+\sigma+iT\right)X^{\sigma+iT}d\sigma \ll \frac{X^{\epsilon}}{T^{A}},
\end{split}
\end{align}
so that 
\begin{equation}
\lim_{T \rightarrow \infty}\int_{H_{1}} f(\alpha)d\alpha = 0,
\end{equation}
and similarly 
\begin{equation}
\lim_{T \rightarrow \infty}\int_{H_{2}} f(\alpha)d\alpha = 0.
\end{equation}
Finally, we bound

\begin{align}
\begin{split}
\lim_{T \rightarrow \infty} \int_{V_{2}} f(\alpha)d\alpha &= -i \int_{\R}h\left(-\frac{1}{4}+\epsilon+it\right)\tilde{\Phi}\left(\frac{1}{4}+\epsilon+it\right)X^{(-\frac{1}{4}+\epsilon+it)}dt\\
&\ll \int_{\R} \textnormal{min}(1,|t|^{-A})X^{-\frac{1}{4}+\epsilon}X^{it}dt \ll X^{-\frac{1}{5}},
\end{split}
\end{align}
from which the theorem then follows.
\end{proof}

\begin{lemma}\label{holomorphic integrals} 
Let $\alpha, \beta, \tilde{\Phi}$ be as above.  Suppose $h(\alpha,\beta)$ is holomorphic\footnote{A function $f:\Omega \subset \C^{2} \mapsto \C$ is said to be \textit{homolorphic} if it is holomorphic in each variable separately.} in the region 
\begin{equation}
\Omega \times \Omega := \left\{(\alpha,\beta): -\frac{1}{4}< \textnormal{Re}(\alpha),\textnormal{Re}(\beta)< \epsilon\right\}
\end{equation}
for some $\epsilon >0$, and moreover that $h(\alpha,\beta)$ does not grow too rapidly in $\Omega \times \Omega$, i.e., does not grow too rapidly in each variable, separately.  Then
\begin{equation}
\int_{(\epsilon')}\tilde{\Phi}\left (\frac 1 2+\beta \right)\int_{(\epsilon)}h(\alpha,\beta)\tilde{\Phi}\left ( \frac 1 2+\alpha \right)X^{\alpha+\beta} ~d\alpha d\beta \ll X^{-\frac{2}{5}}.
\end{equation}
\end{lemma}
\begin{proof}

Set
\begin{equation}
f_{\beta}(\alpha) := h_{\beta}(\alpha)\tilde{\Phi}\left(\frac 1 2+\alpha \right)X^{\alpha},
\end{equation}

where $h_{\beta}(\alpha):=h(\alpha,\beta)$. Since $f_{\beta}$ is holomorphic, by an application of Lemma \ref{countour bound} we write
\begin{equation}
\int_{(\epsilon)} f_{\beta}(\alpha)d\alpha =O_{\beta}\left(X^{-\frac{1}{5}}\right)=O\left(g(\beta)\cdot X^{-\frac{1}{5}}\right),
\end{equation}
where $g$ does not grow too rapidly as a function of $\beta$.  By another application of Lemma \ref{countour bound}, it then follows that

\begin{align}
\begin{split}
\int_{(\epsilon')}\tilde{\Phi}\left ( \frac 1 2+\beta \right)X^{\beta} \bigg(\int_{(\epsilon)}f_{\beta}(\alpha)~d\alpha\bigg) d\beta &\ll \int_{(\epsilon')}g(\beta) \tilde{\Phi}\left (\frac 1 2+\beta \right)X^{-\frac{1}{5}+\beta} ~d\beta\\
&\ll X^{-\frac{2}{5}}.
\end{split}
\end{align}
\end{proof}

\begin{lemma}\label{holomorphic bound 2}
Let $\alpha, \beta$,$\tilde{\Phi}$, and $f_{\beta}$ be as above.  Suppose $f_{\beta}(\alpha)$ has a finite pole at $a_{k}(\beta)$ with residue $\textnormal{Res}(f_{\beta},a_{k}(\beta))$.  Moreover, suppose that for each $a_{k}(\beta)$, $\textnormal{Res}(f_{\beta},a_{k}(\beta))$ is holomorphic in $\Omega := \left\{-\frac{1}{4}< \textnormal{Re}(\beta)< \epsilon\right\}$ for some $\epsilon > 0$, and that $\textnormal{Res}(f_{\beta},a_{k}(\beta))$ does not grow too rapidly in $\Omega$.  Then
\begin{equation}
\int_{(\epsilon')}\tilde{\Phi}\left ( \frac 1 2+\beta \right)X^{\beta}\int_{(\epsilon)}f_{\beta}(\alpha) ~d\alpha d\beta \ll X^{-\frac{1}{5}}.
\end{equation}
\end{lemma}
\begin{proof}
By Lemma \ref{countour bound}, we write
\begin{equation}\label{inner error}
\int_{(\epsilon)} f_{\beta}(\alpha)d\alpha =2\pi i\cdot \sum_{k}\textnormal{Res}(f_{\beta},a_{k}(\beta))+O\left(g(\beta)\cdot X^{-\frac{1}{5}}\right),
\end{equation}
where, as in the proof of Lemma \ref{holomorphic integrals}, we explicitly note the dependence of the error term on $\beta$. Applying Lemma \ref{holomorphic integrals} to the error term in (\ref{inner error}), we obtain
\begin{align}
\int_{(\epsilon')}\tilde{\Phi}\left ( \frac 1 2+\beta \right)X^{\beta} &\int_{(\epsilon)}f_{\beta}(\alpha)~d\alpha d\beta \\
&= 2\pi i \cdot \int_{(\epsilon')}\tilde{\Phi}\left ( \frac 1 2+\beta \right)\sum_{k}\textnormal{Res}(f_{\beta},a_{k}(\beta))X^{\beta} d\beta+O\left(X^{-\frac{2}{5}}\right),
\end{align}
and finally by another application of Lemma \ref{countour bound},
\begin{align}
\int_{(\epsilon')}\tilde{\Phi}\left ( \frac 1 2+\beta \right)\sum_{k}\textnormal{Res}(f_{\beta},a_{k}(\beta))X^{\beta} d\beta &\ll X^{-\frac{1}{5}}.
\end{align}
\end{proof}
\begin{lemma} Let $C_{\Phi}$ and $C'_{\Phi}$ be as in $(\ref{mean value constants})$ and $(\ref{C constant 1})$, respectively.  Then
\begin{equation}\label{C constant 2}
C_{\Phi} = - 2\pi i \int_{(\epsilon')}\tilde{\Phi}\left(\frac 1 2+\beta \right)\tilde{\Phi}\left(\frac 1 2-\beta \right)~d\beta
\end{equation}

and
\begin{equation}\label{C' constant 2}
C'_{\Phi} = - 2 \pi i\int_{(\epsilon')}\tilde{\Phi}\left(\frac 1 2+\beta \right)\tilde{\Phi}'\left(\frac 1 2-\beta \right)~d\beta.
\end{equation}
\end{lemma}
\begin{proof}
Set $\phi(y) = \Phi(e^{y})e^{y/2}$ so that

\begin{align}
\tilde{\Phi}\left(\frac 1 2 +it \right) = \int_{0}^{\infty}\Phi(x)x^{-\frac  1 2 +it}dx= \int_{\R} \phi(y)e^{iy t}dy = \widehat{\phi}\left(-\frac{t}{2\pi}\right),
\end{align}
and similarly $\tilde{\Phi}\left(\frac 1 2 -it \right) =  \widehat{\phi}\left(\frac{t}{2\pi}\right).$  By shifting the integral to Re$(\beta) = 0$ we obtain
\begin{align}
2\pi i \int_{(\epsilon')}\tilde{\Phi}\left(\frac 1 2+\beta\right)\widetilde{\Phi}\left(\frac{1}{2}-\beta\right)~d\beta= 2\pi i \int_{\R}\widehat{\phi}\left(-\frac{t}{2\pi}\right)\widehat{\phi}\left(\frac{t}{2\pi}\right)i ~dt.
\end{align}
Since $\overline{\widehat{\phi}\left(-\frac{t}{2\pi}\right)}=\widehat{\phi} \left(\frac{t}{2\pi}\right)$, we moreover have that

\begin{align}
\begin{split}
\int_{\R}\widehat{\phi}\left(-\frac{t}{2\pi}\right)\widehat{\phi}\left(\frac{t}{2\pi}\right)i dt&= \int_{\R}\left|\widehat{\phi}\left(-\frac{t}{2\pi}\right)\right|^2 i dt= 2\pi i\cdot \int_{\R}\left|\widehat{\phi}\left(x\right)\right|^2 dx\\
&= 2\pi i\cdot \int_{0}^{\infty}\Phi(x)^2 dx,
\end{split}
\end{align}
i.e.,
\begin{equation}
C_{\Phi} = 4 \pi^2 \int_{0}^{\infty}\Phi(x)^2 dx = -2\pi i \int_{(\epsilon')}\tilde{\Phi}\left(\frac 1 2+\beta\right)\widetilde{\Phi}\left(\frac{1}{2}-\beta\right)~d\beta.
\end{equation}
Next, note that
\begin{align}
\begin{split}
\tilde{\Phi}'\left(\frac 1 2-\beta \right)=-\frac{d}{d\beta}\tilde{\Phi}\left(\frac 1 2-\beta\right)&=-\frac{d}{d\beta}\int_{0}^{\infty}\Phi(x)x^{\frac{1}{2}-\beta-1}dx\\
&=\int_{0}^{\infty}\Phi(x)(\log x) x^{-\beta -\frac{1}{2}}dx.
\end{split}
\end{align}
Upon setting $g(y) = y \cdot \Phi(e^{y})e^{y/2}$, we write
\begin{align}
\int_{0}^{\infty}\Phi(x)(\log x) x^{ -\frac{1}{2}-it}dx= \int_{\R}g(y)e^{-iyt}dy =\widehat{g}\left(\frac {t}{2\pi}\right),
\end{align}
so that by shifting to the half-line Re$(\beta) = 1/2$, it follows that
\begin{align}
\begin{split}
2 \pi i\int_{(\epsilon')}\tilde{\Phi}\left(\frac 1 2+\beta \right)\tilde{\Phi}'\left(\frac 1 2-\beta \right)~d\beta &= 2\pi i \int_{\R}\widehat{g}\left(\frac {t}{2\pi}\right)\widehat{\phi}\left(-\frac {t}{2\pi}\right)i ~dt\\
&= (2\pi i)^{2} \cdot \int_{\R}\widehat{g}(x)\overline{\widehat{\phi}(x)} ~dx\\
&= -4\pi ^2 \cdot \int_{\R}g(x)\overline{\phi(x)} ~dx\\
&= -4\pi ^2 \cdot \int_{0}^{\infty}\log x \cdot \Phi(x)^2 ~dx.
\end{split}
\end{align}
\end{proof}

\section{Proof of Lemma \ref{1st total}}
In this section we seek to compute

\begin{equation}\label{first integral}
I_{1}= \int_{(\epsilon')}\int_{(\epsilon)}\frac{\partial }{\partial \beta}\frac{\partial}{\partial \alpha} G(\alpha,\beta,\gamma,\delta)\bigg \vert_{(\alpha,\beta,\alpha,\beta)}\tilde \Phi\left(\frac{1}{2}+\alpha\right)\tilde \Phi\left(\frac{1}{2}+\beta\right)X^{\alpha+\beta}d\alpha d\beta.
\end{equation}
  Note that

\begin{align}\label{dadb Y}
\begin{split}
\frac{\partial}{\partial \alpha}\frac{\partial}{\partial\beta}&G(\alpha,\beta,\gamma,\delta)\bigg \vert_{(\alpha,\beta,\alpha,\beta)} = \frac{\partial}{\partial \alpha}\frac{\partial}{\partial\beta}\bigg (\mathcal{Y}(\alpha,\beta,\gamma,\delta)\cdot \mathcal{A}(\alpha,\beta,\gamma,\delta)\bigg )\bigg \vert_{(\alpha,\beta,\alpha,\beta)}\\
&=\frac{\zeta''}{\zeta}(1+\alpha+\beta)-\frac{\zeta'}{\zeta}(1+\alpha+\beta)^2+\frac{\zeta'}{\zeta}(1+2\alpha)\frac{\zeta'}{\zeta}(1+2\beta)\\
&\phantom{=}+\frac{\zeta'}{\zeta}(1+2\alpha)\cdot\frac{\partial}{\partial \beta}\mathcal{A}(\alpha,\beta,\gamma,\delta)\bigg|_{(\alpha,\beta,\alpha,\beta)}\\
    & \phantom{=}+\frac{\zeta'}{\zeta}(1+2\beta)\cdot \frac{\partial}{\partial \alpha}\mathcal{A}(\alpha,\beta,\gamma,\delta)\bigg|_{(\alpha,\beta,\alpha,\beta)}+\frac{\partial}{\partial \alpha}\frac{\partial}{\partial \beta}\mathcal{A}(\alpha,\beta,\gamma,\delta)\bigg|_{(\alpha,\beta,\alpha,\beta)},
\end{split}
\end{align}

where we recall that $\tilde{\mathcal{A}}(\alpha,\beta,\alpha,\beta)=1$. Since
\begin{equation}
h(\alpha,\beta) := \frac{\partial}{\partial \alpha}\frac{\partial}{\partial \beta}A(\alpha,\beta,\gamma,\delta)\bigg|_{(\alpha,\beta,\alpha,\beta)}
\end{equation}
is holomorphic in $\Omega \times \Omega$,  by Lemma \ref{holomorphic integrals} we find that the integral corresponding to this term is bounded by $O\left(X^{-2/5}\right)$.  Moreover, by an application of Lemma \ref{holomorphic bound 2}, the integrals corresponding to
\begin{equation}
\frac{\zeta'}{\zeta}(1+2\alpha)\cdot \frac{\partial}{\partial \beta}A(\alpha,\beta,\gamma,\delta)\bigg|_{(\alpha,\beta,\alpha,\beta)}\hspace{5mm} \textnormal{and} \hspace{5mm} \frac{\zeta'}{\zeta}(1+2\beta)\cdot \frac{\partial}{\partial \alpha}A(\alpha,\beta,\gamma,\delta)\bigg|_{(\alpha,\beta,\alpha,\beta)}
\end{equation}
are each bounded by $O\left(X^{-1/5}\right)$. The main contributions to (\ref{first integral}) thus come from
\begin{equation}
\frac{\zeta''}{\zeta}(1+\alpha+\beta), \hspace{5mm} -\frac{\zeta'}{\zeta}(1+\alpha+\beta)^2, \hspace{5mm} \textnormal{and } \hspace{5mm} \frac{\zeta'}{\zeta}(1+2\alpha)\cdot \frac{\zeta'}{\zeta}(1+2\beta),
\end{equation}
and we now proceed to separately compute each of the three corresponding integrals.

\subsection{Computing $\frac{\zeta''}{\zeta}(1+\alpha+\beta):$}\label{1st}
The first double integral we would like to compute is

\begin{align}\label{firstintegral}
\begin{split}
I_{(\ref{1st})}&:=\int_{(\epsilon')}\int_{(\epsilon)}\frac{\zeta''}{\zeta}(1+\alpha+\beta)\tilde{\Phi}\left(\frac 1 2+\alpha\right)\tilde{\Phi}\left(\frac 1 2+\beta\right)X^{(\alpha+\beta)} ~d\alpha~d\beta\\
&=\int_{(\epsilon')}\tilde{\Phi}\left(\frac 1 2+\beta\right)X^{\beta} \int_{(\epsilon)}f_{(\ref{1st})}(\alpha) ~d\alpha ~d\beta,
\end{split}
\end{align}
 
where
\begin{equation}
f_{(\ref{1st})}(\alpha) := \frac{\zeta''}{\zeta}(1+\alpha+\beta)\tilde{\Phi}\left(\frac 1 2+\alpha\right)X^{\alpha}.
\end{equation}

Since $f_{(\ref{1st})}$ has one double pole at $\alpha=-\beta$, it follows from Lemma \ref{countour bound} that

\begin{align}
\int_{(\epsilon)}f_{(\ref{1st})}(\alpha) ~d\alpha &= 2\pi i \cdot \text{Res}(f_{(\ref{1st})},-\beta)+O\left(X^{-\frac{1}{5}}\right).
\end{align}
To compute $\text{Res}(f_{(\ref{1st})},-\beta)$, we split $f_{(\ref{1st})}(\alpha)$ into two parts.\\
\\
i) First, we expand $\frac{\zeta''}{\zeta}(1+\alpha+\beta)$ about the point $\alpha = -\beta$, yielding

\begin{equation}
\frac{\zeta''}{\zeta}(1+\alpha+\beta) = \frac{2}{(\alpha+\beta)^{2}}-\frac{2\gamma_{0}}{(\alpha+\beta)}+2(\gamma_0^2+\gamma_1)+h.o.t.,
\end{equation}
where $\gamma_i$ are Stieltjes constants, not to be confused with the variable $\gamma$ used previously.\\

ii) Next, we expand $g(\alpha) = \tilde{\Phi}\left(\frac 1 2+\alpha\right)X^{\alpha}$ about the point $\alpha = -\beta$.  Since
\begin{equation}
g'(\alpha) =  \tilde{\Phi}\left(\frac 1 2+\alpha \right)(\log X) X^{\alpha}+\frac{d}{d\alpha}\tilde{\Phi}\left(\frac 1 2+\alpha \right)X^{\alpha},
\end{equation}

it follows that
\begin{align}\label{g expansion}
\begin{split}
g(\alpha)&= \tilde{\Phi}\left(\frac 1 2-\beta \right)X^{-\beta}+\bigg(\tilde{\Phi}\left(\frac 1 2-\beta\right)(\log X) X^{-\beta}\\
&\phantom{=}+\tilde{\Phi}'\left(\frac 1 2-\beta\right)X^{-\beta}\bigg)(\alpha+\beta) +h.o.t.
\end{split}
\end{align}

Multiplying the two Taylor expansions above, we find that

\begin{align}
\begin{split}
\textnormal{Res}(f_{(\ref{1st})},-\beta)&= 2\bigg(\tilde{\Phi}\left(\frac 1 2-\beta\right)(\log X) +\tilde{\Phi}'\left(\frac 1 2-\beta\right)-\gamma_{0}\tilde{\Phi}\left(\frac 1 2-\beta \right)\bigg)X^{-\beta},
\end{split}
\end{align}

and therefore
\begin{align*}
\begin{split}
\int_{(\epsilon)}f_{(\ref{1st})}(\alpha) ~d\alpha&=4\pi i \bigg(\tilde{\Phi}\left(\frac 1 2-\beta \right)(\log X) +\tilde{\Phi}'\left(\frac 1 2-\beta \right)-\gamma_{0}\tilde{\Phi}\left(\frac 1 2-\beta \right)\bigg)X^{-\beta}\\
&\phantom{=}+O\left(X^{-\frac{1}{5}}\right).
\end{split}
\end{align*}

By an application of Lemma \ref{countour bound}, it follows that

\begin{align}
\begin{split}
&I_{(\ref{1st})} =4\pi i \bigg(\log X \int_{(\epsilon')}\tilde{\Phi}\left(\frac 1 2+\beta \right) \tilde{\Phi}\left(\frac 1 2-\beta \right)~d\beta\\
&+\int_{(\epsilon')}\tilde{\Phi}\left(\frac 1 2+\beta \right) \tilde{\Phi}'\left(\frac 1 2-\beta \right)~d\beta-\gamma_{0}\int_{(\epsilon')}\tilde{\Phi}\left(\frac 1 2+\beta \right)\tilde{\Phi}\left(\frac 1 2-\beta \right)~d\beta\bigg)\\
&+O\left(X^{-\frac{2}{5}}\right),
\end{split}
\end{align}
i.e.,
\begin{equation}\label{1st; 1}
\boxed{I_{(\ref{1st})} = - 2 (\log X) C_{\Phi}-2 C'_{\Phi}+2 \gamma_{0} C_{\Phi}+O\left(X^{-\frac{2}{5}}\right).}
\end{equation}

\subsection{Computing $-\frac{\zeta'}{\zeta}(1+\alpha+\beta)^2$}\label{2nd}

Next, we are interested in the integral
\begin{align}\label{integral3}
\begin{split}
I_{(\ref{2nd})}&:=-\int_{(\epsilon')}\int_{(\epsilon)}\frac{\zeta'}{\zeta}(1+\alpha+\beta)^2\cdot \tilde{\Phi}\left(\frac 1 2+\alpha \right)\tilde{\Phi}\left(\frac 1 2+\beta \right)X^{\alpha+\beta} ~d\alpha~d\beta\\
&\hspace{1mm}=-\int_{(\epsilon')}X^{\beta}\tilde{\Phi}\left(\frac 1 2+\beta \right)\int_{(\epsilon)}f_{(\ref{2nd})}(\alpha) ~d\alpha~d\beta,
\end{split}
\end{align}

where
\begin{equation}
f_{(\ref{2nd})}(\alpha) := \frac{\zeta'}{\zeta}(1+\alpha+\beta)^2\cdot \tilde{\Phi}\left(\frac 1 2+\alpha \right)X^{\alpha}.
\end{equation}
Since $f_{(\ref{2nd})}(\alpha)$ has a single pole at $\alpha = -\beta$, it follows from Lemma \ref{countour bound} that

\begin{equation}
\int_{(\epsilon)}f_{(\ref{2nd})}(\alpha) ~d\alpha = 2\pi i \cdot \textnormal{Res}(f_{(\ref{2nd})},-\beta) +O\left(X^{-\frac 1 5}\right).
\end{equation}

To determine the residue of this integral at the point $\alpha = -\beta$, we expand $\frac{\zeta'}{\zeta}(1+\alpha+\beta)^2$ and $g(\alpha) := \tilde{\Phi}\left(\frac 1 2+\alpha \right)X^{\alpha}$ about the point $\alpha = -\beta$, yielding
\begin{align}
\begin{split}
\frac{\zeta'}{\zeta}(1+\alpha+\beta)^2 &= \frac{1}{(\alpha+\beta)^2}-\frac{2\gamma_{0}}{(\alpha+\beta)}+h.o.t.,
\end{split}
\end{align}
and 
\begin{align}
\begin{split}
g(\alpha)&= \tilde{\Phi}\left(\frac 1 2-\beta \right)X^{-\beta}+\left(\tilde{\Phi}\bigg(\frac 1 2-\beta \right)(\log X) X^{-\beta}\\
&\phantom{=}+\tilde{\Phi}'\left(\frac 1 2-\beta \right)X^{-\beta}\bigg)(\alpha+\beta)+h.o.t.,
\end{split}
\end{align}

so that

\begin{align}
\begin{split}
\textnormal{Res}(f_{(\ref{2nd})},-\beta) &= \bigg(\tilde{\Phi}\left(\frac 1 2-\beta\right)(\log X) +\tilde{\Phi}'\left(\frac 1 2-\beta\right)-2\gamma_{0}\tilde{\Phi}\left(\frac 1 2-\beta \right)\bigg)X^{-\beta}.
\end{split}
\end{align}

It follows that 
\begin{align*}
\begin{split}
\int_{(\epsilon)}f_{(\ref{2nd})}(\alpha) ~d\alpha &= 2\pi i \bigg(\tilde{\Phi}\left(\frac 1 2-\beta \right)(\log X) +\tilde{\Phi}'\left(\frac 1 2-\beta \right)\\
&\phantom{=}-2\gamma_{0}\tilde{\Phi}\left(\frac 1 2-\beta \right)\bigg)X^{-\beta}+O\left(X^{-\frac{1}{5}}\right),
\end{split}
\end{align*}
from which we obtain

\begin{equation}\label{1st; 2}
\boxed{    I_{(\ref{2nd})}=(\log X) C_{\Phi}+C'_{\Phi}-2\gamma_{0} C_{\Phi}+O\left(X^{-\frac{2}{5}}\right).}
\end{equation}
\subsection{Computing $\left(\frac{\zeta'}{\zeta}(1+2\alpha)\right)\left(\frac{\zeta'}{\zeta}(1+2\beta)\right)$}\label{3rd}
Next we are interested in the integral

\begin{align}\label{firstintegral}
\begin{split}
I_{(\ref{3rd})}&:=\int_{(\epsilon')}\int_{(\epsilon)}\frac{\zeta'}{\zeta}(1+2\alpha)\cdot \frac{\zeta'}{\zeta}(1+2\beta)\tilde{\Phi}\left(\frac{1}{2}+\alpha\right)\tilde{\Phi}\left(\frac{1}{2}+\beta\right)X^{\alpha+\beta} ~d\alpha~d\beta\\
&\phantom{:}=\int_{(\epsilon')}f_{(\ref{3rd})}(\beta)~d\beta\cdot \int_{(\epsilon)}f_{(\ref{3rd})}(\alpha)d\alpha,
\end{split}
\end{align}

where
\begin{equation}
f_{(\ref{3rd})}(\alpha) := \frac{\zeta'}{\zeta}(1+2\alpha)\tilde{\Phi}\left(\frac{1}{2}+\alpha\right)X^{\alpha}.
\end{equation}
Since 
\begin{equation}
\frac{\zeta'}{\zeta}(1+2\alpha) = -\frac{1}{2\alpha} + \gamma_{0} + h.o.t.,
\end{equation}

$f$ has a simple pole at $\alpha = 0$ with residue

\begin{align}
\text{Res}(f_{(\ref{3rd})},0) &=\lim_{\alpha \rightarrow 0}\alpha \cdot f_{(\ref{3rd})}(\alpha)=-\frac{1}{2}\tilde{\Phi}\left(\frac{1}{2}\right).
\end{align}

It thus follows from Lemma \ref{countour bound} that 
\begin{equation}
\int_{(\epsilon)}f_{(\ref{3rd})}(\alpha)d\alpha = -\pi i \tilde{\Phi}\left(\frac 1 2\right)+O\left(X^{-\frac 1 5}\right),
\end{equation}
and similarly 
\begin{equation}
\int_{(\epsilon)}f_{(\ref{3rd})}(\beta)d\beta = -\pi i \tilde{\Phi}\left(\frac 1 2\right)+O\left(X^{-\frac 1 5}\right),
\end{equation}
from which we conclude that

\begin{equation}\label{1st; 3}
\boxed{I_{(\ref{3rd})} = -\pi^2\left(\tilde{\Phi}\left(\frac{1}{2}\right)\right)^2+O\left(X^{-\frac 1 5}\right).}
\end{equation}
Lemma \ref{1st total} then follows upon combing the results of (\ref{1st; 1}), (\ref{1st; 2}), and (\ref{1st; 3}).

\section{Proof of Lemma \ref{second total}}
Next, we consider the quantity
\begin{align}\label{2nd ratios piece}
\begin{split}
    &\frac{\partial}{\partial \alpha}\frac{\partial}{\partial\beta}\Bigg(\frac{1}{1-2\alpha} \left(\frac{\pi}{2}\right)^{2\alpha} G(-\alpha,\beta,\gamma,\delta)\Bigg)\Bigg|_{(\alpha,\beta,\alpha,\beta)}=\frac{\zeta(1-2\alpha)}{(1-2\alpha)}\left(\frac \pi 2\right)^{2\alpha}\Bigg(\mathcal{A}(-\alpha,\beta,\alpha,\beta)\\
    &\bigg(-\frac{\zeta'}{\zeta}(1+2\beta)-\frac{\zeta'}{\zeta}(1-\alpha+\beta)+\frac{\zeta'}{\zeta}(1+\alpha+\beta)\bigg)-\frac{d}{d\beta}\mathcal{A}(\alpha,\beta,\gamma,\delta)\bigg|_{(-\alpha,\beta,\alpha,\beta)}\Bigg)
\end{split} 
   \end{align}
coming from the integral $I_{2}$, as well as the symmetric quantity
\begin{align}\label{3rd ratios piece}
\begin{split}    &\frac{\partial}{\partial \alpha}\frac{\partial}{\partial\beta}\bigg(\frac{1}{1-2\beta} \left(\frac{\pi}{2}\right)^{2\beta} G(\alpha,-\beta,\gamma,\delta)\bigg)\bigg|_{(\alpha,\beta,\alpha,\beta)}=\frac{\zeta(1-2\beta)}{(1-2\beta)}\left(\frac{\pi}{2}\right)^{2\beta}\Bigg(\mathcal{A}(\alpha,-\beta,\alpha,\beta)\\
&\bigg(-\frac{\zeta'}{\zeta}(1+2\alpha)-\frac{\zeta'}{\zeta}(1+\alpha-\beta)+\frac{\zeta'}{\zeta}(1+\alpha+\beta)\bigg)-\frac{\partial}{\partial \alpha}\mathcal{A}(\alpha,\beta,\gamma,\delta)\bigg|_{(\alpha,-\beta,\alpha,\beta)}\Bigg)
\end{split}
   \end{align}
coming from the integral $I_{3}$. As before, we approach this term by term, and note that by an application of Lemma \ref{holomorphic bound 2}, the integrals over

\begin{equation}
\frac{d}{d\beta}\mathcal{A}(\alpha,\beta,\gamma,\delta)\bigg|_{(-\alpha,\beta,\alpha,\beta)} \hspace{5mm} \textnormal{and} \hspace{5mm}  \frac{\partial}{\partial \alpha}\mathcal{A}(\alpha,\beta,\gamma,\delta)\bigg|_{(\alpha,-\beta,\alpha,\beta)} 
\end{equation}
may be bounded by $O\left(X^{-\frac 1 5}\right)$. Significant contributions then come from integration against the following integrands:\\
\\
$i) -\frac{\zeta'}{\zeta}(1+2\beta)$ and $ -\frac{ \zeta'}{\zeta}(1+2\alpha)$,\\
\\
$ii)-\frac{ \zeta'}{ \zeta}(1-\alpha+\beta)$ and $-\frac{ \zeta'}{ \zeta}(1+\alpha-\beta)$,\\
\\
$iii) \hspace{2mm} 2 \cdot \frac{\zeta'}{ \zeta}(1+\alpha+\beta)$.
\subsection{Computing $ -\frac{ \zeta'}{\zeta}(1+2\beta)$ and $ -\frac{ \zeta'}{\zeta}(1+2\alpha)$:} \label{4th}
Combining the discussion above with (\ref{full conjecture}), we seek to compute the following integral:

\begin{align}
I_{(\ref{4th})}&:=-\int_{(\epsilon)}\frac{\zeta(1-2 \alpha)}{(1-2\alpha)} \cdot \left(\frac{\pi}{2}\right)^{2\alpha}\tilde{\Phi}\left(\frac 1 2+\alpha\right)X^{\alpha(1-2\lambda)}\bigg (\int_{(\epsilon')}f_{(\ref{4th})}(\beta) d\beta\bigg )~d\alpha,
\end{align}
where
\begin{equation}
f_{(\ref{4th})}(\beta) := \mathcal{A}(-\alpha,\beta,\alpha,\beta)\frac{\zeta'}{\zeta}(1+2 \beta)\tilde{\Phi}\left(\frac 1 2+\beta\right)X^{\beta}.
\end{equation}
Note that since
\begin{equation}
\frac{\zeta'}{\zeta}(1+2\beta) = -\frac{1}{2\beta} + \gamma_{0} + h.o.t.,
\end{equation}

$f_{(\ref{4th})}$ has a simple pole at $\beta = 0$ with residue
\begin{align}
\textnormal{Res}(f_{(\ref{4th})},0) &= -  \mathcal{A}(-\alpha,0,\alpha,0)\frac{1}{2}\tilde{\Phi}\left(\frac 1 2\right),
\end{align}
so that by Lemma \ref{countour bound},

\begin{equation}
\int_{(\epsilon')}f_{(\ref{4th})}(\beta) d\beta = -\pi i \mathcal{A}(-\alpha,0,\alpha,0)\tilde{\Phi}\left(\frac 1 2\right)+O\left(X^{-\frac 1 5}\right).
\end{equation}
Inserting this back into the outer integral, we find that

\begin{align}
I_{(\ref{4th})}&=\pi i \tilde{\Phi}\left(\frac 1 2\right)\int_{(\epsilon)}f'_{(\ref{4th})}(\alpha)~d\alpha+O\left(X^{-\frac 1 5}\right),
\end{align}
where
\begin{equation}
f'_{(\ref{4th})}(\alpha) := \mathcal{A}(-\alpha,0,\alpha,0)\left(\frac \pi 2\right)^{2\alpha} X^{\alpha(1-2\lambda)}\frac{\zeta{(1-2\alpha)}}{(1-2\alpha)}\tilde{\Phi}\left(\frac 1 2+\alpha\right).
\end{equation}
If $\lambda > \frac 1 2$, we shift to the vertical line Re$(\alpha) = 1/5$, so that
\begin{align}
\begin{split}
\int_{(\epsilon)}f'_{(\ref{4th})}(\alpha)~d\alpha &=i \int_{\R}\mathcal{A}\left(-\frac 1 5-it,0,\frac 1 5+it,0\right)\left(\frac{\pi^2 X^{1-2\lambda}}{4}\right)^{\frac 1 5 +it}\\
&\phantom{=}\times \frac{\zeta{(\frac 3 5-2it)}}{(\frac 3 5-2it)}\tilde{\Phi}\left(\frac {7}{10}+it\right)~dt\\
&=i \left(\frac{\pi^2 X^{1-2\lambda}}{4}\right)^{\frac 1 5}\int_{\R}\mathcal{A}\left(-\frac 1 5-it,0,\frac 1 5+it,0\right)\left(\frac{\pi^2 X^{1-2\lambda}}{4}\right)^{it}\\
&\phantom{=}\times\frac{\zeta{\left(\frac 3 5-2it\right)}}{\left(\frac 3 5-2it\right)}\tilde{\Phi}\left(\frac {7}{10}+it\right)~dt.
\end{split}
\end{align}
Since the integrand decays rapidly as a function of $t$, the integral is bounded absolutely by a constant that is independent of $\lambda$.  It follows that for any fixed $\lambda > \frac 1 2$,
\begin{equation}
I_{(\ref{4th})} \ll X^{\left(\frac{1}{5}\right)\left(1-2\lambda\right)}.
\end{equation}
If $\lambda < \frac 1 2 $ we shift to the vertical line Re$(\alpha) = -1/5$, pick up a residue at $\alpha = 0$, and bound the remaining contour by $O\left(X^{\left(-\frac 1 5\right)\left(1-2 \lambda \right)}\right)$.  Since
\begin{equation}
\zeta(1-2\alpha) = -\frac{1}{2\alpha} + \gamma_{0}+h.o.t.,
\end{equation}
the residue is given by
\begin{align}
\textnormal{Res}(f'_{(\ref{4th})},0) &=-\frac{1}{2}\tilde{\Phi}\left(\frac 1 2\right),
\end{align}
where we make use of Lemma \ref{A is 1}.  Since
\begin{equation}
2 \pi i \cdot -\frac{1}{2}\tilde{\Phi}\left(\frac 1 2\right) \pi i \tilde{\Phi}\left(\frac 1 2\right)=\pi^2\left(\tilde{\Phi}\left(\frac 1 2\right)\right)^{2},
\end{equation}
it follows that
\begin{equation}
I_{(\ref{4th})} = \left\{
\begin{array}{l l}
\pi^2\left(\tilde{\Phi}\left(\frac 1 2\right)\right)^{2} +O\left(X^{\left(-\frac 1 5 \right)\left(1-2 \lambda \right)}\right)& \text{ if } \lambda < \frac 1 2 \\
O\left(X^{\left(-\frac{1}{5}\right)\left(2\lambda-1 \right)}\right) & \text{ if } \lambda > \frac 1 2.
\end{array} \right.
\end{equation}

Upon including the contribution from the integral over $ -\frac{ \zeta'}{\zeta}(1+2\alpha)$ coming from the third piece of the Ratios Conjecture, we conclude that the combined contribution from these two symmetric pieces together is equal to 
\begin{equation}\label{3rd; 1}
\boxed{2 \cdot I_{(\ref{4th})} = \left\{
\begin{array}{l l}
2\pi^2\left(\tilde{\Phi}\left(\frac 1 2\right)\right)^{2} +O\left(X^{\left(-\frac 1 5 \right)\left(1-2 \lambda \right)}\right)& \text{ if } \lambda < \frac 1 2 \\
O\left(X^{\left(-\frac{1}{5}\right)\left(2 \lambda-1 \right)}\right) & \text{ if } \lambda > \frac 1 2.
\end{array}\right.}
\end{equation}

\subsection{Computing $-\frac{\zeta'}{ \zeta}(1-\alpha+\beta)$ and $-\frac{\zeta'}{ \zeta}(1+\alpha-\beta)$}\label{5th}
In this section we assume that $0<\textnormal{Re}(\alpha) < \textnormal{Re}(\beta) = \epsilon'$.  The integral that we are interested in computing is 

\begin{align}
I_{(\ref{5th})}&:=-\int_{(\epsilon)}\frac{\zeta(1-2 \alpha)}{(1-2\alpha) }\tilde{\Phi}\left(\frac 1 2+\alpha\right)\left(\frac \pi 2\right)^{2\alpha}X^{\alpha(1-2\lambda)}\bigg (\int_{(\epsilon')}f_{(\ref{5th})}(\beta)~d\beta\bigg ) ~d\alpha,
\end{align}
where
\begin{equation}
f_{(\ref{5th})}(\beta) = \mathcal{A}(-\alpha,\beta,\alpha,\beta)\frac{\zeta'}{ \zeta}(1-\alpha+\beta)\tilde{\Phi}\left(\frac 1 2+\beta\right)X^{\beta}.
\end{equation}
Recalling that

\begin{equation}
\frac{\zeta'}{\zeta}(1-\alpha+\beta) = \frac{1}{\alpha-\beta} + \gamma_{0} + h.o.t.,
\end{equation}
we find that $f_{(\ref{5th})}$ has a simple pole at $\alpha = \beta$.  Under the assumption that $0<\textnormal{Re}(\alpha) < \textnormal{Re}(\beta) = \epsilon'$, this pole is picked up upon shifting the contour to the line Re$(\alpha)=-1/5$, and the residue is 
\begin{align}
\textnormal{Res}(f_{(\ref{5th})},\alpha) &=-\mathcal{A}(-\alpha,\alpha,\alpha,\alpha)\tilde{\Phi}\left(\frac 1 2+\alpha \right)X^{\alpha}.
\end{align}

It follows that
\begin{align}
\begin{split}
I_{(\ref{5th})}&=-\int_{(\epsilon)}\frac{\zeta(1-2 \alpha)}{(1-2\alpha)}\tilde{\Phi}\left(\frac 1 2+\alpha\right)\left(\frac{\pi}{2}\right)^{2\alpha}X^{\alpha(1-2\lambda)}\bigg (-2 \pi i\cdot  \textnormal{Res}(f_{(\ref{5th})},\alpha)+ O\left(X^{-\frac 1 5}\right)\bigg ) ~d\alpha\\
&= 2 \pi i \int_{(\epsilon)}f'_{(\ref{5th})}(\alpha) ~d\alpha+ O\left(X^{-\frac 1 5}\right),
\end{split}
\end{align}
where
\begin{equation}
f'_{(\ref{5th})}(\alpha) = \mathcal{A}(-\alpha,\alpha,\alpha,\alpha)\frac{ \zeta(1-2 \alpha)}{(1-2\alpha) }\left(\frac{\pi}{2}\right)^{2\alpha}X^{2\alpha(1-\lambda)}\tilde{\Phi}\left(\frac 1 2+\alpha\right)\tilde{\Phi}\left(\frac 1 2+\alpha \right).
\end{equation}

If $\lambda >1$, we shift to the vertical line Re$(\alpha) = 1/5$, and bound
\begin{align}
\int_{(\epsilon)}f'_{(\ref{5th})}(\alpha)~d\alpha &= \left(X^{\left(\frac{1}{5}\right)\left(2-2\lambda\right)}\right),
\end{align}

while if $\lambda < 1$, we shift to the vertical line Re$(\alpha) = -\frac{1}{5}$, pick up a pole at $\alpha = 0$, and bound the remaining contour by $O\left(X^{\left(-1/5 \right)\left(2-2 \lambda \right)}\right)$.  Since
\begin{align}
\textnormal{Res}(f'_{(\ref{5th})},0) &= - \frac{1}{2}\tilde{\Phi}\left(\frac 1 2\right)\tilde{\Phi}\left(\frac 1 2 \right),
\end{align}
we conclude that

\begin{equation}\label{3rd; 2}
\boxed{I_{(\ref{5th})} = \left\{
\begin{array}{l l}
2\pi^2\left(\tilde{\Phi}\left(\frac 1 2\right)\right)^{2} +O\left(X^{\left(-\frac 2 5\right)\left(1- \lambda \right)}\right)& \text{ if } \lambda < 1 \\
O\left(X^{\left(-\frac{2}{5}\right)\left(\lambda-1 \right)}\right) & \text{ if } \lambda > 1.
\end{array} \right.}
\end{equation}

Lastly, we consider the integral

\begin{align}
I_{(\ref{5th},\textnormal{sym})}&:=-\int_{(\epsilon')}\frac{\zeta(1-2 \beta)}{(1-2\beta) }\tilde{\Phi}\left(\frac 1 2+\beta\right)\left(\frac \pi 2\right)^{2\beta}X^{\beta(1-2\lambda)}\bigg (\int_{(\epsilon)}f_{(\ref{5th},\textnormal{sym})}(\beta)~d\alpha\bigg ) ~d\beta,
\end{align}
where
\begin{equation}
f_{(\ref{5th},\textnormal{sym})}(\beta) = \mathcal{A}(\alpha,-\beta,\alpha,\beta)\frac{\zeta'}{ \zeta}(1+\alpha-\beta)\tilde{\Phi}\left(\frac 1 2+\alpha\right)X^{\alpha},
\end{equation}

which is the symmetry quantity corresponding to $I_{(\ref{5th},\textnormal{sym})}$ coming from (\ref{3rd ratios piece}) above.  Under the assumption that $0<\textnormal{Re}(\alpha) < \textnormal{Re}(\beta)$, the inner integral is holomorphic in the region $-\frac{1}{5}<\textnormal{Re}(\alpha)<\epsilon$, from which it follows that
\begin{equation}\label{3rd; 2sym}
\boxed{I_{(\ref{5th},\textnormal{sym})} = O\left(X^{-\frac{1}{5}}\right).}
\end{equation}
Note that had we instead assumed $0<\textnormal{Re}(\beta) <\textnormal{Re}(\alpha)<1/5$, we would obtain a significant contribution from $I_{(\ref{5th},\textnormal{sym})}$ and a negligible contribution from $I_{(\ref{5th})}$.  In this way, the symmetry between $\alpha$ and $\beta$ is preserved.
\subsection{Computing $\frac{\zeta'}{ \zeta}(1+\alpha+\beta)$}\label{6th}
Next, we compute

\begin{align}
\begin{split}
I_{(\ref{6th})}&:=\int_{(\epsilon)}\left(\frac{\pi}{2}\right)^{2\alpha}\frac{ \zeta(1-2 \alpha)}{(1-2\alpha)}\tilde{\Phi}\left(\frac 1 2+\alpha\right)X^{\alpha(1-2\lambda)}\left(\int_{(\epsilon')}f_{(\ref{6th})}(\beta)d\beta\right)~d\alpha,
\end{split}
\end{align}
where 
\begin{equation}
f_{(\ref{6th})}(\beta) = \mathcal{A}(-\alpha,\beta,\alpha,\beta)\frac{\zeta'}{ \zeta}(1+\alpha+\beta)\tilde{\Phi}\left(\frac 1 2+\beta\right)X^{\beta}.
\end{equation}
Since

\begin{equation}
\frac{\zeta'}{\zeta}(1+\alpha+\beta) = -\frac{1}{\alpha+\beta} + \gamma_{0} + h.o.t.,
\end{equation}
the residue at $\beta = -\alpha$ is

\begin{align}
\textnormal{Res}(f_{(\ref{6th})},-\alpha)&= - \mathcal{A}(-\alpha,-\alpha,\alpha,-\alpha)\tilde{\Phi}\left(\frac 1 2-\alpha\right)X^{-\alpha}.
\end{align}

It follows that
\begin{equation}
\int_{(\epsilon')}f_{(\ref{6th})}(\beta)d\beta = - 2\pi i \mathcal{A}(-\alpha,-\alpha,\alpha,-\alpha)\tilde{\Phi}\left(\frac 1 2-\alpha\right)X^{-\alpha}+O\left(X^{-\frac 1 5}\right),
\end{equation}
and thus upon shifting the line of integration to Re$(\alpha) = 1/5$, we conclude that

\begin{align}\label{3rd; 3rd}
\begin{split}
I_{(\ref{6th})}&=\int_{(\epsilon)}\left(\frac{\pi}{2}\right)^{2\alpha}\frac{\zeta(1-2 \alpha)}{(1-2\alpha)}\tilde{\Phi}\left(\frac 1 2+\alpha\right)X^{\alpha(1-2\lambda)}\left(\textnormal{Res}(f_{(\ref{6th})},-\alpha)+O\left(X^{-\frac 1 5}\right)\right)~d\alpha\\
&=O\left(X^{-\frac {2}{5}\lambda}\right).
\end{split}
\end{align}
Lemma \ref{second total} then follows upon combining the computations in (\ref{3rd; 1}), (\ref{3rd; 2}), (\ref{3rd; 2sym}), and (\ref{3rd; 3rd}).
\section{Proof of Lemma \ref{fourth total}}\label{7th}
Since
\begin{align}\label{4th integral}
    &\frac{\partial}{\partial \alpha} \frac{\partial}{\partial\beta}\bigg( \frac{1}{1-2(\alpha+\beta)} \left(\frac{\pi}{2}\right)^{2(\alpha+\beta)} G(-\alpha,-\beta,\gamma,\delta)\bigg )\bigg \vert_{(\alpha,\beta,\alpha,\beta)}=  \\
    &\frac{\zeta(1-2\alpha)\zeta(1-2\beta)}{(1-2(\alpha+\beta))}\left(\frac \pi 2\right)^{2(\alpha+\beta)}\Bigg(\frac{\zeta(1-\alpha-\beta)\zeta(1+\alpha+\beta)}{\zeta(1+\alpha-\beta)\zeta(1-\alpha+\beta)}\Bigg)\mathcal{A}(-\alpha,-\beta,\alpha,\beta),\nonumber
\end{align}
we write
\begin{align}\label{outer 4th}
I_{4}&=\int_{(\epsilon')}\zeta(1-2\beta)\tilde{\Phi}\left(\frac 1 2+\beta \right)\left(\frac{\pi}{2}\right)^{2\beta}X^{\beta(1 - 2\lambda)} \bigg(\int_{(\epsilon)}f_{4}(\alpha) ~d\alpha\bigg)~d\beta,
\end{align}
where
\begin{align}
\begin{split}
f_{4}(\alpha) &= \mathcal{A}(-\alpha,-\beta,\alpha,\beta)\left(\frac{\pi}{2}\right)^{2\alpha}X^{\alpha(1 - 2\lambda)}\frac{ \zeta(1-2 \alpha)}{(1-2(\alpha+\beta))}\\
&\phantom{=}\times \Bigg(\frac{\zeta(1-\alpha-\beta)\zeta(1+\alpha+\beta)}{\zeta(1+\alpha-\beta)\zeta(1-\alpha+\beta)}\Bigg)\tilde{\Phi}\left(\frac 1 2+\alpha\right).
\end{split}
\end{align}

Suppose $\lambda > 1/2$.  We then shift to the vertical line Re$(\alpha) = 1/5$, so that
\begin{align}
\begin{split}
\int_{(\epsilon)}f_{4}(\alpha) ~d\alpha &=i \left(\frac{\pi^2 X^{1-2\lambda}}{4}\right)^{\frac 1 5}\int_{\R}\mathcal{A}\left(-\frac 1 5-it,-\beta,\frac 1 5+it,\beta\right)\left(\frac{\pi^2 X^{1-2\lambda}}{4}\right)^{it}\\
&\phantom{=}\times \frac{\zeta{(\frac 3 5-2it)}}{(\frac 3 5-2it-2\beta)}\Bigg(\frac{\zeta(\frac 4 5-it-\beta)\zeta\left(\frac 6 5+it+\beta\right)}{\zeta\left(\frac 6 5+it-\beta\right)\zeta\left(\frac 4 5-it+\beta\right)}\Bigg)\tilde{\Phi}\left(\frac {7}{10}+it\right)~dt.
\end{split}
\end{align}
By the decay properties of $\Phi$, the integral is bounded by a constant (depending on $\beta$) that is independent of $\lambda$.  It follows that
\begin{equation}
\int_{(\epsilon)}f_{4}(\alpha) ~d\alpha = O_{\beta}\left(X^{\frac{1}{5}\left(1-2\lambda\right)}\right) = O\left(g(\beta)\cdot X^{\frac{1}{5}\left(1-2\lambda\right)}\right),
\end{equation}
where $g$ does not grow too rapidly as a function of $\beta$.
Inserting this back into the outer integral, and shifting the line of integration to Re$(\beta) = 1/5$, we obtain
\begin{align}
I_{4} &\ll X^{\frac{1}{5}\left(1-2\lambda\right)} \cdot \int_{(\epsilon')}g(\beta)\cdot\zeta(1-2\beta)\tilde{\Phi}\left(\frac 1 2+\beta \right)\left(\frac{\pi}{2}\right)^{2\beta}X^{\beta(1 - 2\lambda)} d\beta\ll X^{\frac{2}{5}\left(1-2\lambda\right)}.
\end{align}
Next, suppose $\lambda < 1/2$.  We shift the line of integration to $\textnormal{Re}(\alpha) = - 1/5$, and pick up a simple at $\alpha=0$, and a double pole at $\alpha = - \beta$.  By an application of Lemma \ref{countour bound}, we then find
\begin{equation}
\int_{(\epsilon)} f_{4}(\alpha)d\alpha =2\pi i\cdot \bigg(\textnormal{Res}(f_{4},0)+\textnormal{Res}(f_{4},-\beta)\bigg)+O\left( X^{-\frac{1}{5}(1-2\lambda)}\right).
\end{equation}
It remains to compute these two residue contributions.
\subsection{Simple Pole at $\alpha = 0$:}\label{8th}
Note that $f_{4}$ has a simple pole at $\alpha = 0$ with residue
\begin{align}
\textnormal{Res}(f_{4},0) &= -\frac{1}{2}\mathcal{A}(0,-\beta,0,\beta) \frac{1}{(1-2\beta)}\tilde{\Phi}\left(\frac 1 2\right),
\end{align}
which contributes when $\lambda < 1/2$.  Inserting this into the outer integral, we find that
\begin{align}\label{outer simple 4th}
\begin{split}
&\int_{(\epsilon')}\zeta(1-2\beta)\tilde{\Phi}\left(\frac 1 2+\beta \right)\left(\frac{\pi}{2}\right)^{2\beta}X^{\beta(1 - 2\lambda)}\left(-\pi i \mathcal{A}(0,-\beta,0,\beta)\frac{1}{(1-2\beta)}\tilde{\Phi}\left(\frac 1 2\right)\right)d\beta\\
&=-\pi i\tilde{\Phi}\left(\frac 1 2\right)\int_{(\epsilon')}f_{(\ref{8th})}(\beta)d\beta,
\end{split}
\end{align}

where
\begin{align}
f_{(\ref{8th})}(\beta) &= \frac{\zeta(1-2\beta)}{(1-2\beta)}\tilde{\Phi}\left(\frac 1 2+\beta \right)\left(\frac{\pi}{2}\right)^{2\beta}X^{\beta(1 - 2\lambda)}\mathcal{A}(0,-\beta,0,\beta).
\end{align}

The integral in (\ref{outer simple 4th}) has a simple pole at $\beta = 0$ with residue
\begin{align}
\textnormal{Res}(f_{(\ref{8th})},0) &=-\frac{1}{2}\tilde{\Phi}\left(\frac 1 2\right),
\end{align}
so that the total contribution from this pole is 
\begin{equation}\label{simple pole contribution}
\boxed{-\pi^2 \left(\tilde{\Phi}\left(\frac 1 2\right)\right)^{2}+ O\left(X^{-\frac 1 5(1-2\lambda)}\right).}
\end{equation}
\subsection{Double Pole at $\alpha = -\beta$:}
To compute the residue of $f_{4}$ at the point $\alpha = - \beta$, we split $f_{4}(\alpha)$ into three components.\\
\\
i) First, define
\begin{equation}
h(\alpha):= \mathcal{A}(-\alpha,-\beta,\alpha,\beta)\frac{\zeta(1-2 \alpha)}{\zeta(1+\alpha-\beta)\zeta(1-\alpha+\beta)}\frac{\tilde{\Phi}\left(\frac 1 2+\alpha\right)}{(1-2(\alpha+\beta))}.
\end{equation}
Since $h(\alpha)$ is holomorphic at $\alpha=-\beta$, we may expand it as a power series of the form
\begin{align}
h(\alpha) &=h(-\beta)+ h^{(1)}(-\beta)(\alpha+\beta)+h.o.t.
\end{align}

ii) Next, we expand
\begin{equation}
\left(\frac \pi 2 \right)^{2\alpha}\left( X^{1 - 2\lambda}\right)^{\alpha} = e^{\alpha (\log (\frac {\pi^2}{4}) +(1 - 2\lambda)\log X)}=e^{\alpha\cdot C}
\end{equation}
 about the point $\alpha = -\beta$, where
\begin{equation}
C := \log \left(\frac {\pi^2}{4}\right) +(1 - 2\lambda)\log X.
\end{equation}
The expansion is given as

\begin{equation}
e^{\alpha\cdot C}=e^{-\beta\cdot C}+C\cdot e^{-\beta\cdot C}(\alpha+\beta) +h.o.t.
\end{equation}

iii) Finally, we note that 
\begin{equation}
\zeta(1-\alpha-\beta)\zeta(1+\alpha+\beta) = \left(-\frac{1}{\alpha+\beta} + \gamma_{0}+h.o.t.\right)\left(\frac{1}{\alpha+\beta} + \gamma_{0}+h.o.t.\right).
\end{equation}

The total residue is then found to be the full coefficient of $(\alpha+\beta)^{-1}$, i.e.,
\begin{align}
\begin{split}
\textnormal{Res}(f_{4},-\beta) &=-C\cdot e^{-\beta\cdot C}h(-\beta) -e^{-\beta\cdot C}h^{(1)}(-\beta).
\end{split}
\end{align}
We now compute these two contributions separately.
\subsubsection{First Piece}
The total contribution from the first piece is 

\begin{align}
&-2\pi i \cdot \left(\log \left(\frac {\pi^2}{4}\right) +(1 - 2\lambda)\log X\right)\left(\frac{\pi}{2}\right)^{-2\beta} X^{-\beta(1 - 2\lambda)}\cdot \frac{\tilde{\Phi}\left(\frac 1 2-\beta\right)}{\zeta(1-2\beta)},
\end{align}
where we note that $\mathcal{A}(\beta,-\beta,-\beta,\beta)= 1$.  Inserting this into the outer integral of (\ref{outer 4th}), we find that the main contribution of this piece is

\begin{align}
-2\pi i \int_{(\epsilon')}\tilde{\Phi}\left(\frac 1 2+\beta \right)\tilde{\Phi}\left(\frac 1 2-\beta \right)\cdot \left(\log \left(\frac {\pi^2}{4}\right) +(1 - 2\lambda)\log X\right)~d\beta,
\end{align}
i.e., the total contribution is given by
\begin{equation}\label{first piece}
\boxed{C_{\Phi}\left(\log \left(\frac {\pi^2}{4}\right) +(1 - 2\lambda)\log X\right)+ O\left(X^{-\frac 1 5(1-2\lambda)}\right).}
\end{equation}
\subsubsection{Second Piece}
One directly computes

\begin{align}
\begin{split}
h^{(1)}(-\beta)&=\frac{1}{\zeta(1-2\beta)}\Bigg(\tilde{\Phi}'\left(\frac 1 2-\beta\right)+\tilde{\Phi}\left(\frac 1 2-\beta\right)\bigg(2-\frac{\zeta'}{\zeta}(1-2\beta) -\frac{\zeta'}{\zeta}(1+2\beta)\\
&\hspace{5mm}+A_{\beta}'(-\beta)+\frac{L'}{L}(1-2
\beta)+\frac{L'}{L}(1+2\beta)\bigg)\Bigg),
\end{split}
\end{align}
upon noting that $A_{\beta}(-\beta)=A(\beta,-\beta,-\beta,\beta)=1.$  Inserting this expression back into the outer integral of (\ref{outer 4th}), we find that the total contribution from this piece is 
\begin{equation}\label{second piece}
\boxed{2 C_{\Phi}+C_{\Phi,\zeta}-C_{\Phi,L} +C'_\Phi -A_{\Phi}'+O\left(X^{-\frac 1 5(1-2\lambda)}\right),}
\end{equation}
where
\begin{equation}\label{zeta constant}
C_{\Phi,\zeta} :=2\pi i \int_{(\epsilon')}\tilde{\Phi}\left(\frac 1 2+\beta \right)\tilde{\Phi}\left(\frac 1 2-\beta \right)\left(\frac{\zeta'}{\zeta}(1-2\beta)+\frac{\zeta'}{\zeta}(1+2\beta)\right)~d\beta,
\end{equation}
\begin{equation}\label{L constant}
C_{\Phi,L} := 2\pi i \int_{(\epsilon')}\tilde{\Phi}\left(\frac 1 2+\beta \right)\tilde{\Phi}\left(\frac 1 2-\beta \right)\left(\frac{L'}{L}(1+2\beta)+\frac{L'}{L}(1-2\beta)\right)~d\beta,
\end{equation}
and
\begin{equation}\label{A constant}
A_{\Phi}' := - 4\pi i \int_{(\epsilon')} \tilde{\Phi}\left(\frac 1 2+\beta \right)\tilde{\Phi}\left(\frac 1 2-\beta \right)\Bigg(\sum_{\substack{p\equiv 3(4)\\p \textnormal{ prime}}}\frac{\left(p^{2+8\beta}+p^2-2 p^{4\beta}\right) \log p}{p^{2+8\beta}+p^2-p^{4\beta}-p^{4+4\beta}}\Bigg)d\beta,
\end{equation}
where we have made use of Lemma \ref{A integral}.  Lemma \ref{fourth total} then follows upon combining the results of (\ref{simple pole contribution}), (\ref{first piece}), and (\ref{second piece}).

\appendix

\section{Obtaining Numerical Evidence for Conjecture \ref{conj}}

The data provided in Figure \ref{abillionprimes} was obtained using the Mathematica code provided below.  Fix $X = 10^9$, $\Phi = 1_{(0,1]}$, and $f = 1_{[-\frac{1}{2},\frac{1}{2}]}$.  The code outputs $\mathrm{Var}(\psi_{K,X})/(\langle \psi_{K,X} \rangle \log X)$ as a function of $\lambda:= \log K/\log X$, for values of $\lambda$ ranging between $0.1 \leq \lambda \leq 0.7$ with step size $0.025$.  For simplicity, we ignore the small contributions coming from prime powers, as well as from the unique prime $( 1+i ) \subset \Z[i]$ lying above 2.

\fontsize{10}{12}\selectfont

\begin{mmaCell}{Code}
X = 10^9; (* This size took a long time for Mathematica to run.*)
A = 1; (* We count primes in Z[i] with norm from A to B *)
B = X;

Roundmod[\mmaPat{m_}, \mmaPat{res_}, \mmaPat{N_}] = Ceiling[(\mmaPat{m} - \mmaPat{res})/\mmaPat{N}]*\mmaPat{N} + \mmaPat{res}; (* An auxiliary function which finds the smallest integer n >= m such that n=res (mod N). *)

gauss = Take[
   Ratios[Flatten[
     Table[PowersRepresentations[p, 2, 2], {p, 
       Select[Range[Roundmod[A, 1, 4], B, 4], PrimeQ]}]]], {1, -1, 
    2}];
    (* For primes p which are 1 modulo 4 between the specified ranges A and B, we compute the unique representation p = a^2 + b^2 for a,b nonnegative integers with a < b. Then we return the list of numbers b/a *)
gauss2 = Table[N[ArcTan[theta]], {theta, gauss}]; (* Using the list "gauss" we compute the angles associated to Gaussian primes (a + bi) lying over a rational prime congruent to 1 modulo 4, for 0 <= a < b. *)
gauss3 = Table[N[ArcTan[theta]], {theta, Table[1/gauss[[i]], {i, Length[gauss]}]}]; (* Using the list "gauss" we compute the angles associated to Gaussian primes (a + bi) lying over a rational prime congruent to 1 modulo 4, for 0 <= b < a. These are complex conjugates of the primes giving angles in the "gauss2" list. *)

primes1 = Select[Range[Roundmod[A, 1, 4], B, 4], PrimeQ]; (* We find the primes which are 1 modulo 4, between the ranges A and B. *)
primes3 = Select[Range[Roundmod[Sqrt[A], 3, 4], Sqrt[B], 4], PrimeQ]; (* We find the primes which are 3 modulo 4, between the ranges A and B. *)
trivial = Table[0., Length[primes3]]; (* The rational primes which are 3 modulo 4 remain prime in the Gaussian integers, and have an angle of zero. This list contains one zero for each prime congruent to 3 modulo 4, between A and B. *)

allAngles = Join[trivial, gauss2, gauss3]; (* This is a list, with multiplicity, of the angles of Gaussian primes with norm between A and B. By convention, the angle is in the interval [0,Pi/2). *)

allPrimes = N[Join[2 Log[primes3], Log[primes1], Log[primes1]]]; (* The elements of this list correspond to Gaussian primes P with norm between A and B. The Gaussian prime P appears as the number log(N(P)), which is the von Mangoldt function evaluated at P. Suppose P lies over a rational prime p. If p is 3 modulo 4 then N(P) = p^2, and P is the unique Gaussian prime lying over p. If p is 1 modulo 4, then we have N(P)=p and there is exactly one other Gaussian prime P' lying over the same prime p. *)

anglesWeights = WeightedData[allAngles, allPrimes];
Do[Print[{j, 
    Divide[Variance[
      Last[HistogramList[
        anglesWeights, {0, Divide[Pi, 2], 
         Divide[Pi, 2 Round[X^j]]}]]], X^{1 - j}*Log[X]]}], {j, .1, 
  .7, .025}] (* This outputs pairs {lambda, Var(psi_{K,X})/(<psi_{K,X}> log(X))} for .1 <= lambda <= .7, with step size .025 for lambda.*)
\end{mmaCell}

\fontsize{11}{12}\selectfont

The following is used to compute a numerical approximation for $C_{\Phi,\zeta}$, when $\Phi = 1_{(0,1]}$:

\fontsize{10}{12}\selectfont

\begin{mmaCell}{Code}
<< NumericalCalculus` (* imports a package that allows us to take numerical limits and derivatives *)
PhiTilde[\mmaPat{s_}] := (1/\mmaPat{s}) (* Mellin transform of Phi. *)
PhiTildeProduct[\mmaPat{t_}] := PhiTilde[1/2 + I*\mmaPat{t}]*PhiTilde[1/2 - I*\mmaPat{t}]
ZetaPrime[\mmaPat{s_}] := ND[Zeta[t], t, \mmaPat{s}] (* Using Mathematica's in-built Zeta function. We take a derivative *)

2*Pi*I*(I*
   NIntegrate[
    PhiTildeProduct[t]*(ZetaPrime[1 + .2 + 2*I*t]/Zeta[1 + .2 + 2*I*t] + 
       ZetaPrime[1 - .2 - 2*I*t]/Zeta[1 - .2 - 2*I*t]), {t, -25, 25}])
\end{mmaCell}

\fontsize{11}{12}\selectfont

The following is used to compute a numerical approximation for $C_{\Phi,L}$, when $\Phi = 1_{(0,1]}$:

\fontsize{10}{12}\selectfont

\begin{mmaCell}{Code}
<< NumericalCalculus` (* imports a package that allows us to take numerical limits and derivatives *)
PhiTilde[\mmaPat{s_}] := (1/\mmaPat{s}) (* Mellin transform of Phi. *)
PhiTildeProduct[\mmaPat{t_}] := PhiTilde[1/2 + I*\mmaPat{t}]*PhiTilde[1/2 - I*\mmaPat{t}]
L[\mmaPat{s_}] := N[DirichletL[4, 2, \mmaPat{s}]] (* This is the Dirichlet L-function for the non-trivial character modulo 4. *)
LPrime[\mmaPat{s_}] := ND[L[t], t, \mmaPat{s}] (* Takes a derivative of the L-function. *)

-2*Pi*I*(I*
   NIntegrate[
    PhiTildeProduct[t]*(LPrime[1 + .2 + 2*I*t]/L[1 + .2 + 2*I*t] + 
       LPrime[1 - .2 - 2*I*t]/L[1 - .2 - 2*I*t]), {t, -25, 25}])
\end{mmaCell}

\fontsize{11}{12}\selectfont

The following is used to compute a numerical approximation for $A'_{\Phi}$, when $\Phi = 1_{(0,1]}$:

\fontsize{10}{12}\selectfont

\begin{mmaCell}{Input}
h[\mmaPat{β_}, \mmaPat{p_}] = \mmaFrac{ (p^2 - 2 p^(4 \mmaUnd{β}) + p^(2 + 8 \mmaUnd{β})) Log[p] }{ p^2 - p^(4 \mmaUnd{β}) - p^(4 + 4 \mmaUnd{β}) + p^(2 + 8 \mmaUnd{β}) }
\end{mmaCell}

\begin{mmaCell}{Code}
PhiTilde[\mmaPat{s_}] := (1/\mmaPat{s}) (* Mellin transform of Phi. *)
PhiTildeProduct[\mmaPat{t_}] := PhiTilde[1/2 + I*\mmaPat{t}]*PhiTilde[1/2 - I*\mmaPat{t}]

qn[\mmaPat{p_}] = NIntegrate[h[I*\mmaPat{t}, p]*PhiTildeProduct[\mmaPat{t}], {\mmaPat{t}, 0, Infinity}]
primes3 = Select[Range[3, 1000, 4], PrimeQ]; (* Selects the primes congruent to 3 modulo 4 which are less than 1000. *)

output = 0;
For[i = 1, i <= Length[primes3], i++,
 output += 2*qn[primes3[[i]]]]
   Print["Range is ", j, ". Integral is ", output]
\end{mmaCell}

\end{document}